\documentclass[leqno,11pt]{amsart}
\usepackage{amssymb, amsmath,amsmath,latexsym,amssymb,amsfonts,amsbsy, amsthm,esint}
\usepackage{color}
\usepackage{graphicx}
\usepackage{tikz}
\usetikzlibrary{arrows, calc, decorations.markings,intersections, patterns,patterns.meta}
\usepackage{caption}
\usepackage{subcaption}
\usepackage{hyperref}

\usepackage{pgfplots}
\pgfplotsset{compat=1.15}
\usepackage{mathrsfs}

\makeatletter
\DeclareFontFamily{U}{tipa}{}
\DeclareFontShape{U}{tipa}{m}{n}{<->tipa10}{}
\newcommand{\arc@char}{{\usefont{U}{tipa}{m}{n}\symbol{62}}}%

\newcommand{\arc}[1]{\mathpalette\arc@arc{#1}}

\newcommand{\arc@arc}[2]{%
  \sbox0{$\m@th#1#2$}%
  \vbox{
    \hbox{\resizebox{\wd0}{\height}{\arc@char}}
    \nointerlineskip
    \box0
  }%
}
\makeatother

\let\pa=\partial
\let\al=\alpha
\let\b=\beta
\let\g=\gamma
\let\d=\delta

\let\lam=\lambda

\let\s=\sigma
\let\t=\theta
\let\f=\frac

\let\G= \Gamma
\let\D=\Delta

\let\Om=\Omega

\let\e=\varepsilon
\let\pa=\partial

\let\ri=\rightarrow
\let\na=\nabla

\newcommand{\beq}{\begin{equation}}
\newcommand{\eeq}{\end{equation}}
\newcommand{\beqo}{\begin{equation*}}
\newcommand{\eeqo}{\end{equation*}}
\newcommand{\ben}{\begin{eqnarray}}
\newcommand{\een}{\end{eqnarray}}
\newcommand{\beno}{\begin{eqnarray*}}
\newcommand{\eeno}{\end{eqnarray*}}

\newtheorem{theorem}{Theorem}[section]
\newtheorem{definition}[theorem]{Definition}
\newtheorem{lemma}[theorem]{Lemma}
\newtheorem{proposition}[theorem]{Proposition}
\newtheorem{corol}[theorem]{Corollary}

\theoremstyle{remark}

\newtheorem{case}{Case}[section]
\newtheorem{rmk}{Remark}[section]

\newcommand{\dist}{\mathrm{dist}}

\newcommand{\BR}{\mathbb{R}}

\newcommand{\ch}{\mathcal{H}^1}
\newcommand{\cl}{\mathcal{L}^1}

\newcommand{\tu}{\tilde{u}}

\newcommand{\mres}{\mathbin{\vrule height 1.6ex depth 0pt width 0.13ex\vrule height 0.13ex depth 0pt width 1.1ex}}

\begin{document}

\title[Uniqueness of the blow-down limit]{Uniqueness of the blow-down limit for triple junction problem}

\author{Zhiyuan Geng}
\address{Department of Mathematics, Purdue University, 150 N. University Street, West Lafayette, IN 47907–2067}
\email{geng42@purdue.edu}

\begin{abstract}
    We prove the uniqueness of $L^1$ blow-down limit at infinity for an entire minimizing solution $u:\BR^2\ri\BR^2$ of a planar Allen-Cahn system with a triple-well potential. Consequently, $u$ can be approximated by a triple junction map at infinity. The proof exploits a careful analysis of energy upper and lower bounds, ensuring that the diffuse interface remains within a small neighborhood of the approximated triple junction at all scales.   
\end{abstract}

\date{\today}
\maketitle

\section{Introduction} \label{sec: intro}

\subsection{The problem and main result}

This paper is concerned with the uniqueness of the blow-down limit at infinity for an entire, bounded minimizing solution of the system 
\begin{equation}
    \label{main eq} \D u-W_u(u)=0,\quad u:\BR^2\ri\BR^2.
\end{equation}
where $W$ is a potential with three global minima. Specifically for $W$ we assume
\vspace{3mm}
\begin{enumerate}
    \item[(H1).]  $W\in C^2(\BR^2;[0,+\infty))$, $\{z: \,W(z)=0\}=\{a_1,a_2,a_3\}$, $W_u(u)\cdot u >0$ if $|u|>M$ and 
    \beqo
     c_2|\xi^2| \geq \xi^TW_{uu}(a_i)\xi\geq  c_1|\xi|^2,\; i=1,2,3.
    \eeqo 
    for some positive constants $c_1<c_2$ depending on $W$. 
\item[(H2).] For $i\neq j$, $i,j\in \{1,2,3\}$, let $U_{ij}\in W^{1,2}(\BR,\BR^2)$ be an 1D minimizer of the \emph{action}
\beqo
\sigma_{ij}:=\min \int_{\BR}\left(\f12|U_{ij}'|^2+W(U_{ij})\right)\,d\eta, \quad \lim\limits_{\eta\ri-\infty}U_{ij}(\eta)=a_i,\ \lim\limits_{\eta\ri+\infty}U_{ij}(\eta)=a_j.
\eeqo
$\sigma_{ij}$ satisfies
\beq\label{cond on sigma}
\sigma_{ij}\equiv \sigma>0\quad \text{ for }i\neq j\in\{1,2,3\} \text{ and some constant }\sigma.
\eeq
\end{enumerate}

Note that \eqref{main eq} is the Euler-Lagrange equation associated with the energy
\begin{equation}
    \label{main energy} E(u,\Omega):=\int_{\Omega} \left(\f12|\na u|^2+W(u)\right)\,dx.
\end{equation}

An entire \emph{minimizing solution} $u: \BR^2\ri \BR^2$ is defined in the following local sense. 
\begin{definition}\label{def: minimizing sol}
A function $u:\BR^2\ri\BR^2$  is a \emph{minimizing solution} of \eqref{main eq} in the sense of De Giorgi if 
\begin{equation}
    E(u,\Om)\leq E(u+v,\Om), \quad \forall \text{ bounded open sets } \Om\subset \BR^2, \ \forall  v\in C_0^1(\Om).
\end{equation}
\end{definition}

The solution we seek can be regarded as the diffuse analogue of the minimal 3--partition of the plane. Specifically, we define
\beqo
\mathcal{P}:=\{\mathcal{D}_1,\mathcal{D}_2,\mathcal{D}_3\}, 
\eeqo
which is a partition of the plane into three sectors, which are centered at the origin with opening angles of $120$ degrees. Let $\pa\mathcal{P}$ denote the union of three rays that separate $\{\mathcal{D}_i\}_{i=1}^3$, i.e.
$$
\pa \mathcal{P}:=\pa \mathcal{D}_1\cup\pa \mathcal{D}_2\cup\pa \mathcal{D}_3.
$$
We call $\partial \mathcal{P}$ a \emph{triod}. The \emph{triple junction map} is defined by 
\beq\label{form of triple junction sol}
U_{\mathcal{P}}:= a_1\chi_{\mathcal{D}_1}+ a_2\chi_{\mathcal{D}_2}+a_3\chi_{\mathcal{D}_3},
\eeq
where $\chi_{\Om}$ represents the characteristic equation of domain $\Om$.

It is well known that $\mathcal{P}$ is a minimizing partition of $\mathbb{R}^2$ into three phases and $\pa\mathcal{P}$ is a minimal cone. $\mathcal{P}$ is minimizing in the sense that for any $\Omega\subset \BR^2$, $\mathcal{P}\mres \Omega=\{D_i\cap\Om\}_{i=1}^3 $ is a solution of the following variational problem 
\beqo
\begin{split}
&\qquad\qquad\quad \min \sum\limits_{i<j} \mathcal{H}^1(\pa (A_i\cap \Om) \cap \pa(A_j\cap\Om)),\\
&\mathcal{A}=\{A_i\}_{i=1}^3 \text{ is a 3--partition of }\mathbb{R}^2 \text{ and } \mathcal{P}\mres(\mathbb{R}^2\setminus \Omega)= \mathcal{A}\mres(\mathbb{R}^2\setminus \Omega). 
\end{split}
\eeqo

The minimality stated above is related to Steiner's classical result which states that given three points $A,B,C$ on the plane such that the corresponding triangle has no angle greater than or equal to $120$ degrees, then if $P$ is a point that minimizes the sum of the distances $\vert P-A\vert+\vert P-B\vert+\vert P-C\vert$, the line segments $PA,PB,PC$ form three $120$-degree angles. 

In 1996, Bronsard, Gui and Schatzman \cite{bronsard1996three} established the existence of an entire solution in the equivariant class of the reflection group $\mathcal{G}$ of the symmetries of the equilateral triangle. The triple-well potential is also assumed invariant under $\mathcal{G}$. The solution is obtained as a minimizer in the equivariant class $u(gx)=gu(x), g\in \mathcal{G}$, hence is not necessarily stable under general perturbations. Their results was extended to the three dimensional case in 2008 by Gui and Schatzman \cite{gui2008symmetric}. We refer to the book \cite{afs-book} and to the references therein. 

In 2021, Fusco \cite{fusco} succeeded in establishing essentially the result of \cite{bronsard1996three} in the equivariant class of the rotation subgroup of $\mathcal{G}$ (by $\f23\pi$), thus eliminating the two reflections. The symmetry hypothesis fixes the center of the junction at the origin, which simplifies the analysis. 

On bounded domains, there are some constructions of triple junction solutions without imposing symmetry assumptions, see for example the paper by Sternberg and Ziemer \cite{sternberg1994local} for clover-shaped domains in $\BR^2$ via $\Gamma$--convergence, and for more general domains by Flores, Padilla and Tonegawa \cite{flores2001higher} by the a mountain pass argument.  These results do not allow to pass to the limit and establish the existence of a triple junction solution on $\BR^2$.

The existence of an entire minimizing solution in the sense of Definition \ref{def: minimizing sol}, characterized by a triple junction structure at infinity, has been independently established in two recent papers: by Alikakos and the author \cite{alikakos2024triple} and by Sandier and Sternberg \cite{sandier2023allen}. Under slightly different hypotheses and employing distinct methods, these two studies have yielded comparable results saying that along a subsequence $r_k\ri\infty$, the rescaled function $u_{r_k}(z):=u(r_kz)$ converges in $L^1_{loc}(\BR^2)$ to a triple junction map $u_0$ of the form \eqref{form of triple junction sol}. However, it remains unclear a priori if there could exist two different sequences of rescalings $\{r_k\}$ and $\{s_k\}$, leading to distinct blow-down limits $u_1$ and $u_2$, corresponding to distinct minimal 3--partitions $\mathcal{P}_1$ and $\mathcal{P}_2$, respectively. The primary objective of the present paper is to rule out this possibility and demonstrate the uniqueness of the blow-down limit. We now state our main result.

\begin{theorem}\label{main theorem}
    There exists a bounded, minimizing solution $u:\BR^2\ri\BR^2$ of \eqref{main eq} such that for any compact $K\subset \BR^2$,
    \begin{equation*}
        \lim\limits_{r\ri\infty} \|u_r(x)-u_0\|_{L^1(K)}=0,
    \end{equation*}
    where $u_0=\sum\limits_{i=1}^3 a_i\mathcal{\chi}_{\mathcal{D}_i}$ for $\mathcal{P}=\{\mathcal{D}_i\}_{i=1}^3$ providing a minimal partition of $\BR^2$ into three sectors with angle of $120$ degrees and $\pa \mathcal{P}$ is a triod centered at $0$.
\end{theorem}

\begin{rmk}
    The result above also holds for the case when $\sigma_{ij}$ are not all equal but still satisfy $\sigma_{ij}+\sigma_{jk}<\sigma_{ik}$. In this scenario, the corresponding minimal partition $\mathcal{P}$ consists of three sectors $\mathcal{D}_i$ with opening angles $\theta_i$ satisfying
    $ \sum_{i=1}^3\t_i=2\pi$
 and $\f{\sin{\t_{1}}}{\s_{23}}=\f{\sin{\t_2}}{\s_{13}}=\f{\sin{\t_3}}{\s_{12}}.$ The proof for the general $\sigma_{ij}$ case follows the same argument as the proof of Theorem \ref{main theorem}. 
\end{rmk}

Uniqueness of blow-up or blow-down limits is one of central questions in the study of singular structures in geometric PDEs. Following the prominent early works by Allard and Almgren \cite{allard1981radial} and Simon \cite{simon1983asymptotics}, uniqueness questions have been investigated extensively for free boundary problems, harmonic maps, minimal surfaces and geometric flows. Most of these results rely on some type of the Simon--{\L}ojasiewicz inequality or epiperimetric inequalities, showing the decay of certain monotone quantities at a definite rate. In our proof, the uniqueness of the blow-down limit is obtained from a delicate estimate on the localization of the diffuse interface, derived from a pure variational argument, thus avoid the use of those classical methods.

\subsection{An overview of the proof}
We now list some key steps and ideas in the proof of our results. As mentioned earlier, we can start with a minimizing map $u$ of \eqref{main eq} as constructed in \cite{alikakos2024triple} and \cite{sandier2023allen}, which converges to a triple junction solution of the form \eqref{form of triple junction sol} along some subsequence $r_k\ri\infty$. In particular, for arbitrarily small $\e$, there exists a sufficiently large $R_0$ such that 
\begin{equation*}
    3\sigma-\e \leq \int_{\pa B_{R_0}} \left(\f12|\na_T u|^2+W(u)\right)\,d\ch\leq 3\sigma+\e,
\end{equation*}
where $\na_T$ denotes the tangential derivative. 

Starting from this estimate, which basically means there are three phase transitions along $\pa B_{R_0}$, we can derive that $u(z)$ should ``behave nicely" on $\pa B_{R_0}$. To be more specific, $u$ is close to phases $a_1$, $a_2$, $a_3$ on three arcs $I_1$, $I_2$, $I_3$ respectively. Between these phases, there will be three small arcs $I_{ij}$ ($i\neq j\in \{1,2,3\}$)  separating them, which can be regarded as the place where phase transition happens. We pick points $A\in I_{12}$, $B\in I_{13}$, $C\in I_{23}$ and determine the point $D$ such that $|DA|+|DB|+|DC|$ is minimized. Then we let $T_{ABC}:=DA\cup DB\cup DC $ be the approximate triod on $B_{R_0}$. 

With such a nicely behaved boundary data on $\pa B_{R_0}$, we obtain the following energy upper bound and lower bound in $B_{R_0}$. 
\beq\label{intro: lower upper bdd}
\sigma(|DA|+|DB|+|DC|)-CR_0^\al \leq \int_{B_{R_0}}\left(\f12|\na u|^2+W(u)\right)\,dz\leq \sigma(|DA|+|DB|+|DC|)+CR_0^\al,
\eeq
for some constants $C=C(W)$ and $\al\in (0,1)$. For the upper bound, we utilize the construction of an energy competitor outlined in \cite[Proposition 3.3]{sandier2023allen}, while for the lower bound we mimic the proof of \cite[Proposition 3.1]{alikakos2024triple}. 

We define the diffuse interface as
\beqo
I_{\g}:= \{z:|u(z)-a_i|\geq \g,\ \forall i=1,2,3\}.
\eeqo
The energy bound \eqref{intro: lower upper bdd} implies that $I_\g\cap B_{R_0}$ locates in an $O(R_0^\beta)$ (for some $\beta<1$) neighborhood of the approximate triod $T_{ABC}$. Moreover, away from $T_{ABC}$, the distance of $u(z)$ to $a_i$ will decay exponentially with respect to $\dist(z, T_{ABC})$ thanks to Hypothesis (H1) and standard elliptic theories.

We proceed by rescaling $B_{R_0}$ to the unit disk $B_1$ through the function $u_{R_0}:=u(R_0z)$. The exponential decay above implies that $u_{R_0}$ is closely approximated in $L^1$ norm by a minimal partition $\{\mathcal{D}_{R_0}^i\}_{i=1}^3$ of $B_1$ determined by the rescaled version of $T_{ABC}$. Specifically, we have 
\beq\label{intro: L1 closeness}
\|u_{R_0}-U_{R_0}\|_{L^1}\leq CR_0^{\beta-1},
\eeq
where $U_{R_0}=\sum a_i \chi_{\mathcal{D}_{R_0}^i}$. We point out that \eqref{intro: L1 closeness} holds for any larger scaling $R_i\sim 2^i R_0$, where $U_{R_i}$ is the corresponding approximate triple junction map at the scale $R_i$.

A key observation is that approximate triple junction maps at two consecutive scales are close to each other, i.e.
\begin{equation}\label{intro: closeness of two consecutive scales}
    \|U_{R_i}-U_{R_{i+1}}\|\leq CR_i^{\beta-1}.
\end{equation}
This is established by \eqref{intro: L1 closeness} and the fact that $u_{R_i}$ and $u_{R_{i+1}}$ are obtained by rescaling the same function $u$. Finally, we can iterate \eqref{intro: closeness of two consecutive scales} and deduce that  $U_{R_i}$ will converge to a unique triple junction map, and thereby concluding the proof. 

The article is organized as follows. In Section \ref{sec: preliminaries} we present some preliminary results from \cite{AF} and \cite{afs-book}. In Section \ref{sec: existence} and \ref{sec: bdy data}, we establish the existence of a minimizing solution $u$ and fix a well behaved boundary data on $\pa B_{R_0}$.  Next, we establish the energy bounds \eqref{intro: lower upper bdd} in Section \ref{sec: lower upper bdd}.  The localization of the diffuse interface within an $O(R_0^{\beta})$ neighborhood of the triod $T_{ABC}$ is proved in Section \ref{sec: loc of interface}. Then in Section \ref{sec: rescale} we rescale $B_{R_0}$ to $B_1$ and prove \eqref{intro: L1 closeness}. Lastly, we conclude the proof of the main theorem utilizing the estimate \eqref{intro: closeness of two consecutive scales} in Section \ref{sec: conclusion}.

\section{preliminaries}\label{sec: preliminaries}

Throughout the paper we denote by $z=(x,y)$ a 2D point and by $B_r(z)$ the 2D ball centered at the point $z$ with radius $r$. In addition, we let $B_r$ denote the 2D ball centered at the origin. We recall the following basic results which play a crucial part in our analysis.
\begin{lemma}[Lemma 2.1 in \cite{AF}]\label{lemma: potential energy estimate}
The hypotheses on $W$ imply the existence of $\delta_W>0$, and constants $c_W,C_W>0$ such that
\beqo
\begin{split}
&|u-a_i|=\delta\\
\Rightarrow & \ \f12 c_W\delta^2\leq W(u)\leq \f12 C_W \delta^2,\quad \forall \delta<\delta_W,\ i=1,2,3.
\end{split}
\eeqo
Moreover if $\min\limits_{i=1,2,3} |u-a_i|\geq \delta$ for some $\delta<\delta_W$, then $W(u)\geq \f12 c_W\delta^2.$
\end{lemma}

\begin{lemma}[Lemma 2.3 in \cite{AF}]\label{lemma: 1D energy estimate}
Take $i\neq j \in \{1,2,3\}$, $\d <\d_W$ and $s_+>s_-$ be two real numbers. Let $v:(s_-,s_+)\ri \BR^2$ be a smooth map that minimizes the energy functional 
\beqo
J_{(s_-,s_+)}(v):=\int_{s_-}^{s_+} \left(\f12|\na v|^2+W(v)\right)\,dx 
\eeqo
subject to the boundary condition 
\beqo
|v(s_-)-a_i|=|v(s_+)-a_j|=\delta.
\eeqo
Then
\beqo
J_{(s_-,s_+)}(v)\geq \sigma-C_W\delta^2,
\eeqo
where $C_W$ is the constant in Lemma \ref{lemma: potential energy estimate}. 
\end{lemma}

\begin{lemma}[Variational maximum principle, Theorem 4.1 in \cite{afs-book}]\label{lemma: maximum principle}
     There exists a positive constant $r_0=r_0(W)$ such that for any $u\in W^{1,2}(\Omega,\BR^2)\cap L^\infty(\Om,\BR^2)$ being a minimizer of $E(\cdot,\Om)$, if $u$ satisfies 
    $$
    \vert u(x)-a_i\vert\leq r \text{ on }\pa\Om, \ \ \text{for some }r<r_0,\ i\in\{1,2,3\},
    $$
    then 
    $$
    \vert u(x)-a_i\vert\leq r\ \ \forall x\in \Om.
    $$
\end{lemma}

%\begin{lemma}\label{lemma:ene upper bdd}
%(energy upper bound) \textcolor{blue}{2-phase case and 3-phase case (with Fermat point)}. 
%\end{lemma}

%\section{Proof of Theorem \ref{main theorem}}

\section{Existence of an entire minimizing solution.}\label{sec: existence}

By the constructions in \cite{alikakos2024triple} and \cite{sandier2023allen}, there exists a minimizing solution  $u:\BR^2\ri\BR^2$ of \eqref{main eq} such that the following hold:
    \begin{enumerate}
        \item There exists $M>0$ such that 
        \beq\label{uniform C1 bound}
        |u(z)|+|\na u(z)|\leq M, \quad \forall z\in\BR^2.
        \eeq
        \item For any sequence $r_k\ri\infty$, there is a subsequence, still denoted by $\{r_k\}$, such that 
        \begin{equation}\label{seq conv}
            u(r_k z)\rightarrow u_0(z)\ \text{ in }L^1_{loc}(\BR^2),
        \end{equation}
        where $u_0(z)=\sum\limits_{i=1}^3a_i \mathcal{\chi}_{\mathcal{D}_i}$. for $\mathcal{P}=\{\mathcal{D}_1,\mathcal{D}_2,\mathcal{D}_3\}$ providing a minimal partition of $\BR^2$ into three sectors of the angle $\f23\pi$ and $\pa \mathcal{P}$ is a triod centered at $0$. Moreover, along the same subsequence we have the following energy estimate
        \begin{equation*}
            \lim\limits_{r_k\ri\infty} \f{1}{r_k} E(u,B_{r_k}(0))=3\sigma.
        \end{equation*}
        This energy estimates follows from the $\G$--convergence result in Baldo \cite{Baldo} that holds also without the mass constraint (see Gazoulis \cite{gazoulis}).
        \item By \cite[Lemma 3.4 \& Lemma 3.5]{sandier2023allen}, $u$ is asymptotically homogeneous and satisfies an asymptotic energy equipartition property at large scale,  
        \begin{equation}\label{energy equipartition}
            \lim\limits_{R\ri\infty} \f{1}{R} \int_{B_R}W(u)\,dz=\lim\limits_{R\ri\infty} \f1R\int_{B_R} \f12|\na u|^2\,dz=\lim\limits_{R\ri\infty} \f1R\int_{B_R} \f12|\na_T u|^2\,dz=\frac{3\sigma}{2}. 
        \end{equation}
        Here $\na_T$ denotes the tangential derivative. 
    \end{enumerate}

    Utilizing \eqref{energy equipartition}, for any $\e>0$, there exists a large $R(\e)$ such that for any $R>R(\e)$, 
    \begin{align*}
    3\sigma -\e\leq &\f{1}{R}\int_{B_R} \left( \f12|\na u|^2+W(u)
 \right)\,dz\leq 3\sigma+\e,\\
            \f{3\sigma-\e}{2}\leq &\f{1}{R}\int_{A_{2R,R}}W(u)\,dz\leq \f{3\sigma+\e}{2},\\
            \f{3\sigma-\e}{2}\leq &\f{1}{R}\int_{A_{2R,R}}\f12|\na_T u|^2\,dz\leq \f{3\sigma+\e}{2},
        \end{align*}
    where $A_{2R,R}$ represents the annulus $\{z\in\BR^2: R<|z|<2R\}$. Fixing $\e$ as a small parameter to be determined later, by Fubini's theorem we can find $R_0(\e)\in(R(\e),2R(\e))$ such that 
    \begin{equation}
        \label{energy est on bdy R0} 3\sigma-\e\leq \int_{\pa B_{R_0}}\left(\f12|\na_T u|^2+W(u) \right) \,d\ch\leq 3\sigma+\e.
    \end{equation}
    \begin{equation}
        \label{est:ene in BR0}
        3\sigma-\e\leq \f1{R_0}\int_{ B_{R_0}}\left(\f12|\na u|^2+W(u) \right) \,dz\leq 3\sigma+\e.
    \end{equation}
    As our analysis progresses, we will gradually determine the conditions on the smallness of $\e$.

\section{``well-behaved" profile of $u(z)$ on $\pa {B_{R_0}}$.}\label{sec: bdy data}
Firstly, we take a fixed small constant $\delta=\delta(W)$ which is independent of $\e$, to be determined later. We keep in mind that in certain places of our analysis, $\e$ is considered to be significantly smaller than $\delta$.  From the estimate \eqref{energy est on bdy R0} and Lemma \ref{lemma: potential energy estimate} we conclude that 
    \begin{equation}\label{est: meas of interface on bdy R0}
          \mathcal{H}^1(\{ z: \min\limits_{i=1,2,3}|u(z)-a_i|>\delta\}\cap \pa B_{R_0})\leq \frac{C}{\delta^2},  
    \end{equation}
for some $C=C(W)$. In other words, most of points on $\pa B_{R_0}$ are close to one of the phase $a_i$. We set
\begin{equation*}
    Y_i:=\{z: |u(z)-a_i|\leq \delta\}\cap \pa B_{R_0},\quad i=1,2,3. 
\end{equation*}
Then \eqref{est: meas of interface on bdy R0} implies 
\begin{equation}
    \label{est: meas of Yi}
    \sum\limits_{i=1}^3\ch(Y_i) \geq 2\pi R_0-\f{C}{\delta^2}.
\end{equation}

\begin{lemma}\label{lemma: Yi nonempty}
    For any $i\in\{1,2,3\}$, $Y_i\neq \emptyset$.
\end{lemma}

\begin{proof}
    We need to rule out the following two cases. 
    \begin{case}
        Two of $Y_i$'s are empty sets. Without loss of generality, assume
        \begin{equation*}
             Y_2=Y_3=\emptyset.
        \end{equation*}
        We construct the following energy competitor in $B_{R_0}$ by writing in polar system $z=(x,y)=re^{i\t}$:
        \begin{equation*}
            \tilde{u}(r,\theta)=\begin{cases}
                u(R_0,\theta), &r=R_0\\
             (1+r-R_0)u(R_0,\theta)+(R_0-r)a_1, & r\in(R_0-1,R_0),\\
             a_1,& r\in[0,R_0-1].   
            \end{cases}
        \end{equation*}
    In view of Lemma \ref{lemma: potential energy estimate}, we have that 
    \begin{equation}\label{est: 1 phase, radial der, annulus}
        \begin{split}
            &\int_{A_{R_0,R_0-1}} |\pa_r \tilde{u}|^2\,dz\\
            =&\int_{R_0-1}^{R_0} \int_0^{2\pi} |u(R_0,\theta)-a_1|^2 r\,dr\,d\theta\\
            \leq & \int_{Y_1}|u-a_1|^2\,d\ch +\int_{\pa B_{R_0}\backslash Y_1} |u-a_1|^2\,d\ch\\
            \leq & \int_{Y_1} CW(u)\,d\ch+ \f{C}{\delta^2}M^2\\
            \leq & C(\delta, W). 
        \end{split} 
    \end{equation}
    \begin{equation}\label{est: 1 phase, potential, annulus}
        \int_{A_{R_0,R_0-1}} W(u)\,dz
        \leq \int_{A_{R_0,R_0-1}} C|u-a_1|^2\,dz
        \leq  C(\delta,W).
    \end{equation}
    \begin{equation}\label{est: 1 phase, tan der, annulus}
    \int_{A_{R_0,R_0-1}} |\pa_T\tu|^2\,dz
        \leq \int_{\pa B_{R_0}} |\pa_T \tu|^2\,d\ch
        \leq  3\sigma.
    \end{equation}
    Adding \eqref{est: 1 phase, radial der, annulus}, \eqref{est: 1 phase, potential, annulus} and \eqref{est: 1 phase, tan der, annulus} together gives
    \begin{equation}
        \label{est: 1 phase, BR0}
        \int_{B_{R_0}} \left(\f12|\na u|^2+W(u)\right)\,dz\leq C(\delta,W),
    \end{equation}
    which yields a contradiction with \eqref{est:ene in BR0} given that $R_0$ can be arranged to be much larger than $C(\delta,W)$. As a result, we eliminate Case 1.  
    \end{case}

    \begin{case}
        One of $Y_i$'s is empty. Assume without loss of generality that $Y_3=\varnothing$. Then on $\pa B_{R_0}$, essentially there are only two phases appearing, namely $a_1$ and $a_2$. There exist $z_1,\ z_2\in \pa B_{R_0}$ such that 
        \begin{equation}\label{u at z1 z2}
            |u(z_1)-a_1|\leq \delta,\quad |u(z_2)-a_2|\leq \delta, \quad\text{for } z_i:=R_0e^{i\theta_i},\ i=1,2.
        \end{equation}
        where $\t_i$ denotes the polar angle of $z_i$ ($i=1,2$). Then $z_1$ and $z_2$ split the circle $\pa B_{R_0}$ into two arcs, denoted by $A$ and $B$ respectively. Then we define 
        \begin{align*}
            &A_1:=\{z\in A: |u(z)-a_1|\leq \delta\},\\
            &A_2:=\{z\in A: |u(z)-a_2|\leq \d\}.
        \end{align*}
        By \eqref{u at z1 z2} and the continuity of $u$ we know that $A_1$ and $A_2$ are two disjoint non-empty closed sets in $A$. Set
        \begin{align*}
            {z_A^1}&:= \text{ the point from $A_1$ that is closest to $A_2$},\\
            {z_A^2}&:= \text{ the point from $A_2$ that is closest to $A_1$}.
        \end{align*}
        Here ``closest" refers to the distance along $\pa B_{R_0}$. Now we have the key observation that on $A$ the arcs $\arc{z_1z_A^1}$ and $\arc{z_2z_A^2}$ do not overlap. Suppose by contradiction this is not the case, we have
        \begin{equation*}
            z_A^2\in \arc{z_1z_A^1}\cap A,\quad z_A^1\in \arc{z_2z_A^2}\cap A.
        \end{equation*}
        If we start from $z_1$ and traverse along the clockwise direction of $\partial B_{R_0}$ for a complete circle, we will encounter at least four phase transitions occurring between the two phases $a_1$ and $a_2$, namely 
        \begin{equation*}
            z_1\ri z_2\ri z_A^1\ri z_A^2\ri z_1.
        \end{equation*}
        Then we have 
        \beqo
        \int_{\pa B_{R_0}} \left(\f12|\na_T u|^2+W(u)\right)\,d\ch \geq 4\sigma-C\delta^2,
        \eeqo
        which contradicts \eqref{energy est on bdy R0}. Therefore $\arc{z_1z_A^1}$ and $\arc{z_2z_A^2}$ cannot overlap. Moreover, appealing to \eqref{uniform C1 bound} we have that 
        \beqo
        \cl(\arc{z_A^1z_A^2})\geq C,
        \eeqo
        for some positive constant $C$ depending on $|a_1-a_2|$ and the uniform bound of $|\na u|$. 
        
        Now for any $z\in \arc{z_A^1z_A^2}=A\setminus\left( \arc{z_1z_A^1}\cup \arc{z_2z_A^2} \right)$, it holds
        \beqo
        |u(z)-a_i|>\delta, \quad \forall i\in\{1,2,3\}.
        \eeqo
        Utilizing the energy bound \eqref{energy est on bdy R0} and the hypothesis (H1), we obtain
        \beqo
        \ch(\arc{z_A^1z_A^2})\leq \frac{C}{\delta^2},
        \eeqo
        for some constant $C$ depending only on $W$. 

        In the same manner, on $B$ we can define
        \begin{align*}
            &B_1:=\{z\in B: |u(z)-a_1|\leq \delta\},\\
            &B_2:=\{z\in B: |u(z)-a_2|\leq \d\}.
        \end{align*}
        and 
        \begin{align*}
            {z_B^1}&:= \text{ the point from $B_1$ that is closest to $B_2$},\\
            {z_B^2}&:= \text{ the point from $B_2$ that is closest to $B_1$}.
        \end{align*}
        By the same argument we have $\arc{z_1z_B^1}$ and $\arc{z_2z_B^2}$ do not overlap and 
        \beqo
        C_1\leq \ch(\arc{z_B^1z_B^2})\leq \frac{C_2}{\delta^2},
        \eeqo
        for some positive constants $C_1,\,C_2$. Denote the polar angles for $z_A^1,\,z_A^2,\,z_B^1,\,z_B^2$ by $\t_A^1,\,\t_A^2,\,\t_B^1,\,\t_B^2$ respectively. Set
        \beqo
         \Theta_1:=\arc{\t_B^1 \t_1}\cup\arc{\t_1\t_A^1},\quad \Theta_2:=\arc{\t_A^2 \t_2}\cup\arc{\t_2\t_B^2},\quad \Theta_{0}:=\arc{\t_A^1 \t_A^2}\cup\arc{\t_B^1\t_B^2}.
        \eeqo
        Here $\Theta_i\ (i=1,2)$ approximately represents the set of polar angles for $a_i$ phase points on $\pa B_{R_0}$, where $\Theta_0$ represents the set of polar angles for the transition layer. The size of $\Theta_0$ can be controlled by
        \begin{equation}
            \label{est: size of Theta 0}
            \f{C_1}{R_0}\leq |\Theta_0| \leq \f{C_2}{R_0\delta^2}. 
        \end{equation}
        Now we first define the following function on $\pa B_{R_0-1}$:
        \begin{equation}
            \label{2phase: competitor on bdy B R0-1} \tilde u(R_0-1,\t)=\begin{cases}
                a_1, & \theta \in \Theta_1,\\
                a_2, & \theta\in \Theta_2 ,\\
                \text{smooth connection of }a_1,\,a_2  & \t\in  \Theta_0.
            \end{cases}
        \end{equation}
        Then we extend the energy competitor $\tilde{u}$ to $B_{R_0}$,
        \begin{equation}
            \label{2phase: competitor on B R0}
            \tilde{u}(r,\t):=\begin{cases}
                u(R_0,\t), & r=R_0\\
                (r-R_0+1)u(R_0,\t)+(R_0-r)\tilde{u}(R_0-1,\t) , & r\in(R_0-1,R_0)\\
                \tilde{u}(R_0-1,\t), & r=R_0-1\\
                \text{energy minimizer},  &r\in [0,R_0-1)
            \end{cases}
        \end{equation}
        where we require that $\tilde{u}$ minimizes the energy $E(\cdot, B_{R_0-1})$ with respect to the Dirichlet boundary constraint $\tilde{u}|_{B_{R_0-1}}$.

        We first estimate the energy in the annulus $A_{R_0,R_0-1}$.
        \begin{equation*}
        \begin{split}
            &\int_{A_{R_0,R_0-1}} \left(\f12|\na \tilde{u}|^2+W(\tilde{u})\right)\,dz\\
        =&\left(\int_{A_{R_0,R_0-1}\cap\{\t\in \Theta_1\}} +\int_{A_{R_0,R_0-1}\cap\{\t\in\Theta_2\} }+\int_{A_{R_0,R_0-1}\cap\{\t\in \Theta_0\}}\right) \left(\f12|\na \tilde{u}|^2+W(\tilde{u})\right)\,dz.
        \end{split}
 \end{equation*}
From \eqref{est: size of Theta 0} and \eqref{uniform C1 bound} if follows that 
\beqo
\int_{A_{R_0,R_0-1}\cap\{ \t\in \Theta_0\}} \left(\f12|\na \tilde{u}|^2+W(\tilde{u})\right)\,dz\leq C(\delta,W).
\eeqo

For any $\theta_0\in\Theta_1$, if $|u(R_0,\theta_0)-a_1|< \delta$, then
\beqo
|\tilde{u}(r,\theta_0)-a_1|\leq |u(R_0,\t_0)-a_1|<\delta,\quad \forall r\in(R_0-1,R_0),
\eeqo
which implies 
\beqo
c_1|\tilde{u}(r,\theta_0)-a_1|^2\leq W(\tilde{u}(r,\theta_0))\leq C W(u(R_0,\t_0)), \quad \text{ for some }C=C(W),\ \forall r\in(R_0-1,R_0).
\eeqo
On the other hand if $|u(R_0,\theta_0)-a_1|\geq \delta$, together with the definition of $\Theta_1$ it holds that 
\beqo
|u(R_0,\t_0)-a_i|\geq \delta, \forall i\in\{1,2,3\},
\eeqo
and 
\beqo
\max\{|\tilde{u}(r,\t_0)-a_1|^2,\ W(\tilde{u}(r,\theta_0))\}\leq \frac{C}{\delta^2} W(u(R_0,\t_0)),\quad \forall r\in(R_0-1,R_0).
\eeqo
Following the similar computation as in \eqref{est: 1 phase, radial der, annulus}, \eqref{est: 1 phase, tan der, annulus} and \eqref{est: 1 phase, potential, annulus}, we obtain

\begin{align*}
    &\int_{A_{R_0,R_0-1}\cap \{\t\in\Theta_1\}} |\pa_r \tilde{u}|^2\,dz\\
            =&\int_{R_0-1}^{R_0} \int_{\Theta_1} |u(R_0,\theta)-a_1|^2 r\,dr\,d\theta\\
            \leq & C(\delta,W) R_0\int_{\Theta_1} W(u(R_0,\theta))\,d\t\\
            \leq & C(\delta,W).
\end{align*}

\begin{equation*}
    \int_{A_{R_0,R_0-1}\cap\{\t\in\Theta_1\}} |\pa_T\tu|^2\,dz
        \leq \int_{\pa B_{R_0}} |\pa_T \tu|^2\,d\ch
        \leq  3\sigma.
\end{equation*}

\begin{align*}
&\int_{A_{R_0,R_0-1}\cap\{\t\in \Theta_1\} }W(\tilde{u})\,dz\\
\leq & C(\delta,W) \int_{R_0-1}^{R_0}dr\int_{\Theta_1} W(u(R_0,\theta))r\,d\theta\leq C(\delta,W).
\end{align*}

Summing up the inequalities above implies
\beqo
\int_{A_{R_0,R_0-1}\cap\{\t\in \Theta_1\}} \left(\f12|\na u|^2+W(u)\right)\,dz\leq C(\delta, W). 
\eeqo
This estimate also applies to the energy on $A_{R_0,R_0-1}\cap\{\t\in \Theta_2\}$. Consequently we have 
\begin{equation}\label{ene: 2 phase annulus}
    \int_{A_{R_0,R_0-1}} \left( \f12|\na \tilde{u}|^2+W(\tilde{u}) \right)\,dz\leq C(\delta,W),
\end{equation}

       We are left with the the estimation of $E(\tu,B_{R_0-1})$. Set
        \begin{align*}
            P_1:=(R_0-1)e^{i\f{\t_A^1+\t_A^2}{2}},\quad P_2:=(R_0-1)e^{i\f{\t_B^1+\t_B^2}{2}}.
        \end{align*}
        Let $P_1P_2$ intersects with $\pa B_{R_0-2}$ at two points $P_1'$ and $P_2'$, respectively. Up to a possible rotation, we can assume $P_1P_2$ is parallel to the $x$-axis and their shared $y$ coordinates is denoted by $y_0$. Without loss of generality, we suppose $y_0\leq 0$, with the $a_2$ phase part of $\pa B_{R_0-1}$ positioned above $\{y=y_0\}$, while the $a_1$ phase part is positioned below $\{y=y_0\}$ (see Figure \ref{fig 2phase upper bdd}).

        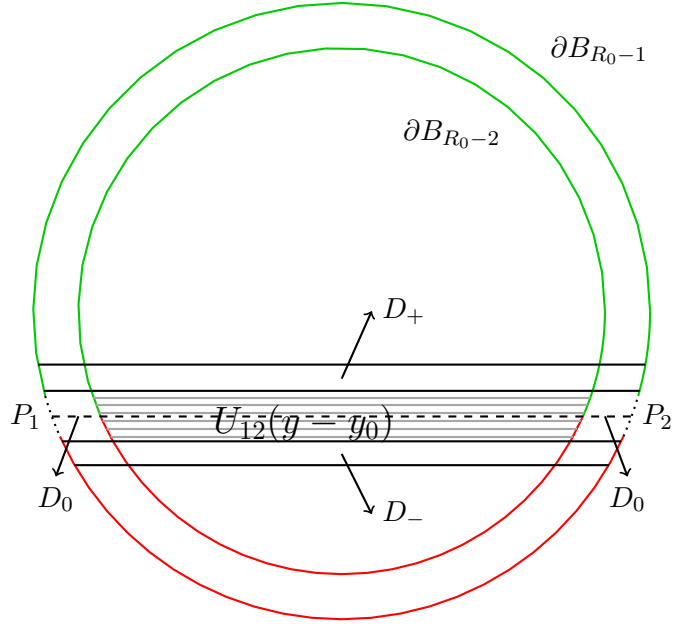
\begin{figure}[ht]
\centering
\begin{tikzpicture}[thick]
\draw [black!20!green,domain=0:204] plot ({3.5*cos(\x)}, {3.5*sin(\x)});
\draw [black!20!green,domain=336:360] plot ({3.5*cos(\x)}, {3.5*sin(\x)});
\draw [red,domain=204:336] plot ({3.5*cos(\x)}, {3.5*sin(\x)});

\draw [black!20!green,domain=0:196] plot ({4.1*cos(\x)}, {4.1*sin(\x)});
\draw [black!20!green,domain=344:360] plot ({4.1*cos(\x)}, {4.1*sin(\x)});
\draw [dotted,domain=196:204] plot ({4.1*cos(\x)}, {4.1*sin(\x)});
\draw [red,domain=204:336] plot ({4.1*cos(\x)}, {4.1*sin(\x)});
\draw [dotted,domain=336:344] plot ({4.1*cos(\x)}, {4.1*sin(\x)});
\draw ({4.1*cos(50)}, {4.1*sin(50)}) node [above right]{$\pa B_{R_0-1}$};
\draw ({3.5*cos(50)}, {3.5*sin(50)}) node [below left]{$\pa B_{R_0-2}$};
\draw[dashed]  ({4.1*cos(200)},{4.1*sin(200)}) node[left]{$P_1$}--({4.1*cos(340)},{4.1*sin(340)})node[right]{$P_2$};
\draw  ({4.1*cos(205)},{4.1*sin(205)})--({4.1*cos(335)},{4.1*sin(335)});
\draw  ({4.1*cos(195)},{4.1*sin(195)})--({4.1*cos(345)},{4.1*sin(345)});
\draw  ({4.1*cos(210)},{4.1*sin(210)})--({4.1*cos(330)},{4.1*sin(330)});
\draw  ({4.1*cos(190)},{4.1*sin(190)})--({4.1*cos(350)},{4.1*sin(350)});
\draw[->] (0,-0.9)--(0.4,0) node[right] {$D_+$};
\draw[->] (0,-1.9)--(0.4,-2.7) node[right] {$D_-$};

\draw[->] (-3.5,-1.4)--(-3.8,-2.2) node[below] {$D_0$};
\draw[->] (3.5,-1.4)--(3.8,-2.2) node[below] {$D_0$};

\path ({3.5*cos(210)},{3.5*sin(210)}) coordinate (M);
\path ({3.5*cos(198)},{3.5*sin(198)}) coordinate (P);
\path ({3.5*cos(330)},{3.5*sin(330)}) coordinate (N);
\fill[pattern={Lines[angle=90,distance=3pt]}, pattern color=gray!70] (M)--(N) arc[start angle=330, end angle=342, radius=3.5]--(P) arc[start angle=198, end angle=210, radius=3.5]  node at (-0.5,-1.5) {\Large{$U_{12}(y-y_0)$}};
 \end{tikzpicture}
\caption{Energy upper bound in $B_{R_0-1}$ for Case 2. Red: $a_1$; Green: $a_2$.}
\label{fig 2phase upper bdd}
\end{figure}

        We pick an 1D minimizing connection $U_{12}$ defined in the Hypothesis (H2). Moreover, it is well-known that $U_{ij}$ will converge to $a_1$ or $a_2$ exponentially (see for example \cite[Proposition 2.4]{afs-book}), 
        \begin{equation*}
            |U_{12}(\eta)-a_2|\leq Ke^{-k\eta},\quad |U_{12}(\eta)-a_1|\leq Ke^{k\eta}, 
        \end{equation*}
        for some constants $K,\,k>0$. Let $h=R_0^{\f13}$. If $y_0\leq -R_0+3h$, elementary geometry implies that $\ch(\arc{z_A^1z_B^1})\sim O(R_0^{\f23})$, which means the $a_1$ phase on $\pa B_{R_0-1}$ has much smaller measure compared to that of the $a_2$ phase. Hence we can invoke similar analysis as in Case 1 to show that  $E(\tu, B_{R_0-1})\leq CR_0^{\f23}\ll O(R_0)$. This, together with \eqref{ene: 2 phase annulus}, yields a contradiction with \eqref{est:ene in BR0}. Therefore, we only consider the case $y_0>-R_0+3h$.  Set
        \begin{equation*}
            v(x,y)=\begin{cases}
            U_{12}(y-y_0), & (x,y)\in \overline{B}_{R_0-2}\cap\{y_0-h\leq y\leq y_0+h\},\\
             a_2, & (x,y)\in B_{R_0-1}\cap\{y\geq y_0+2h\},\\
            a_1,& (x,y)\in B_{R_0-1}\cap \{y\leq y_0-2h\}.
            \end{cases}
        \end{equation*}
        On $B_{R_0-2}\cap\{y=y_0-h\}$ ( or $B_{R_0-2}\cap\{y=y_0+h\}$), by the exponential decay of $U_{12}$ we have 
        \beqo
       |v(x,y)-a_1\,(\text{ or }a_2)|\leq Ke^{-kR_0^{\f13}}.
        \eeqo
        And for $(x,y)\in D_0:= A_{R_0-1,R_0-2}\cap\{y_0-h\leq y\leq y_0+h\}$, we take $v(x,y)$ to linearly interpolate in $x$ between $v|_{\pa B_{R_0-1}}$ and $v|_{\pa B_{R_0-2}}$ for each $y$. Note that $D_0$ consists of two small regions with measure controlled by $O(R_0^{\f13})$. For $(x,y)\in  D_{\pm}:= B_{R_0-1}\cap \{y_0\pm 2h\leq y\leq y_0\pm h\}$, let $v(x,y)$ linearly  interpolate in $y$ between $v(x,y_0\pm 2h)$ and $v(x,y_0\pm h)$. See Figure \ref{fig 2phase upper bdd} for an illustration for all these subdomains. We estimate energy of $v$:
        \begin{eqnarray*}
            E(v,\overline{B}_{R_0-2}\cap\{y_0-h\leq y\leq y_0+h\})&\leq& \sigma *\ch(P_1P_2)\leq 2R_0\sigma,\\
            E(v,B_{R_0-1}\cap \{y\geq y_0+2h\}\cap\{y\leq y_0-2h\}) &=& 0,\\
            E(v,D_0)&\leq& C \mathcal{H}^2(D_0)\leq CR_0^{\f13},\\
            E(v,D_{\pm}) &\leq & C h\cdot\ch(P_1P_2)\cdot K^2e^{-2kR_0^{\f13}}=o(R_0),
        \end{eqnarray*}
        given $R_0$ chosen to be sufficiently large. Altogether,
        \begin{equation*}
        E(v, B_{R_0-1})\leq 2 R_0\sigma+o(R_0). 
        \end{equation*}
        Utilizing the minimality of $\tu$ in $B_{R_0-1}$ and \eqref{ene: 2 phase annulus}, we obtain
\begin{equation*}
    \int_{B_{R_0}}  \left( \f12|\na \tilde{u}|^2+W(\tilde{u}) \right)\,dz\leq 2R_0\sigma+o(R_0),
\end{equation*}
which yields a contradiction with \eqref{energy est on bdy R0} when we select $R_0$ to be large enough. Therefore Case 2 is also eliminated and we get 
\beqo
Y_i\neq \varnothing, \quad \forall i\in\{1,2,3\}.
\eeqo
The proof of Lemma \ref{lemma: Yi nonempty} is complete.  

\end{case}
\end{proof}

From Lemma \ref{lemma: Yi nonempty}, there are $z_1,\,z_2,\,z_3 \in \pa B_{R_0}$ such that 
\beqo
|u(z_i)-a_i|\leq \delta,\quad \forall i\in\{1,2,3\}.
\eeqo

Next we apply the same analysis for the two-phase scenario to obtain the existence of three arcs $\arc{z_l^1z_r^1}$, $\arc{z_l^2z_r^2}$, $\arc{z_l^3z_r^3}$ that satisfies the following properties,
\begin{enumerate}
    \item $z_i\in \arc{z_l^iz_r^i}$, for $i=1,2,3$.
    \item $|u(z_l^i)-a_i|\leq \delta$, $|u(z_r^i)-a_i|\leq \delta$, for $i=1,2,3$.
    \item $\arc{z_l^1z_r^1}$, $\arc{z_l^2z_r^2}$, $\arc{z_l^3z_r^3}$ are disjoint pairwisely. For each $i\in\{1,2,3\}$, we have
    \beqo
    z\in \arc{z_l^iz_r^i} \ \Rightarrow\  |u(z)-a_j|>\delta,\ \ \forall j\neq i.
    \eeqo
    \item Suppose $\arc{z_r^1z_l^2},$ $\arc{z_r^2z_l^3}$, $\arc{z_r^3z_l^1}$ are the three small arcs that separating the arcs $\arc{z_l^1z_r^1}$, $\arc{z_l^2z_r^2}$ and $\arc{z_l^3z_r^3}$. It holds that 
    \beqo
    |u(z)-a_i|>\delta, \quad \forall i\in\{1,2,3\},\ \forall z\in \arc{z_r^1z_l^2}\cup\arc{z_r^2z_l^3}\cup \arc{z_r^3z_l^1},
    \eeqo
    And consequently,
    \beq\label{meas est for transition arcs}
    C_1\leq \ch(\arc{z_r^1z_l^2}\cup\arc{z_r^2z_l^3}\cup \arc{z_r^3z_l^1})\leq \f{C_2}{\delta^2}
    \eeq
\end{enumerate}

We introduce the following notations for convenience.
\begin{align*}
&I_i:= \arc{z_l^iz_r^i} \text{ denotes the set of phase $a_i$ on $\pa B_{R_0}$ }, i=1,2,3;\\
&I_{12}=  \arc{z_r^1z_l^2},\ I_{23}=\arc{z_r^2z_l^3},\ I_{31}=\arc{z_r^3z_l^1} \text{ denote transitional arcs between phases},\\
&d_1:=\min\{\ch(I_{ij}),\, i\neq j\in \{1,2,3\}\},\quad d_2:=\max\{\ch(I_{ij}),\, i\neq j\in\{1,2,3\}\}.
\end{align*}
Then \eqref{meas est for transition arcs} implies 
\beq\label{est d1 d2}
\f{C_1}{3}\leq d_1\leq d_2\leq \f{C_2}{\delta^2}. 
\eeq
By Property (2), there will be an approximate ``phase transition" inside each of the separating arcs $I_{12}$, $I_{23}$, $I_{31}$, and the energy on these arcs is estimated by
\begin{equation}\label{ene est:3 arcs}
\int_{I_{12}\cup I_{23}\cup I_{31}} \left(\f12|\pa_T u|^2+W(u)\right)\,d\ch\geq 3\sigma-C\delta^2,
\end{equation}
for some constant $C=C(W)$. Next we show that on each arc $I_i$, $u(z)$ will be uniformly close to $a_i$, with the distance controlled by $C\delta$.  

\begin{lemma}\label{lemma: on bdy, close to each phase}
    There exists a constant $C$ which only depends on $W$, such that for any $i\in\{1,2,3\}$ and $z\in I_i$, it holds that 
    \beqo
    |u(\theta)-a_i|< C\delta.  
    \eeqo
\end{lemma}

\begin{proof}
Let $C$ be a large constant to be determined later. Without loss of generality, we suppose by contradiction there exists $z^1\in I_1$ such that $|u(z^1)-a_1|= C\delta$ and 
\beqo
|u(z^1)-a_1|=\max\{|u(z)-a_1|, \ z\in I_1\}.
\eeqo
We can choose suitable constants $\delta$ and $C$ satisfying
\beq\label{relation C delta}
C\delta< \delta_W, 
\eeq
where $\delta_W$ is the constant in Lemma \ref{lemma: potential energy estimate}. As a result, for any $z\in \arc{z_l^1z^1}$, it holds that 
\beqo
W(u(z))\geq \f12c_W |u(z)-a_1|^2. 
\eeqo
We compute
\begin{align*}
    &\int_{\arc{z_l^1z^1}} \left( \f12|\pa_T u|^2+W(u) \right)\,d\ch\\
    \geq & \int_{\arc{z_l^1z^1}} \left( \f12|\pa_T |u(z)-a_1||^2+\f12 c_W|u(z)-a_1|^2 \right)\,d\ch\\
    \geq & \int_{\arc{z_l^1z^1}} \left( \sqrt{c_W} \cdot\f12 \pa_T(|u(z)-a_1|^2)  \right)\,d\ch\\
    \geq &\f{\sqrt{c_W}}{2} [(C\delta)^2-\delta^2]\\
    =&\f{\sqrt{c_W}(C^2-1)}{2} \delta^2.
\end{align*}
Combining this with \eqref{ene est:3 arcs}, we can select a sufficiently large $C$ (which does not affect \eqref{relation C delta} as we still have the freedom to choose $\delta$) to obtain 
\beqo
\begin{split}
&\int_{\pa B_{R_0}}\left(\f12|\pa_T u|^2+W(u)\right)\,d\ch\\
\geq & \int_{ I_{12}\cup I_{23}\cup I_{31} \cup \arc{z_l^1z^1}} \left(\f12|\pa_T u|^2+W(u)\right)\,d\ch\\
\geq & 3\sigma +C(W)\delta^2,
\end{split}
\eeqo
which yields a contradiction with \eqref{energy est on bdy R0}, as $\e$ can be chosen arbitrarily small in the beginning. This completes the proof of Lemma \ref{lemma: on bdy, close to each phase}.
\end{proof}

\section{Refined energy upper/lower bound in $B_{R_0}$.}\label{sec: lower upper bdd}

We take $A$, $B$, $C$ to be the middle points of the arcs $I_{12}$, $I_{13}$ and $I_{23}$, respectively. There exists a point $D\in \bar{B}_{R_0}$ such that $|DA|+|DB|+|DC|$ is minimized. It is well-known that if the triangle $ABC$ possesses internal angles that are all less than $\f{2\pi}{3}$, then $D$ is the point in the interior of $ABC$ such that
$$
\angle{ADB}=\angle{BDC}=\angle{ADC}=\f{2\pi}{3}.
$$
However, if one angle, say $\angle BAC$ is greater than or equal to $\f{2\pi}{3}$, then $D$ coincide with the vertex $A$ of this triangle.

We define the triod
\begin{equation*}
    T_{ABC}=DA\cup DB\cup DC.
\end{equation*} 
Then we invoke the energy upper bound established in \cite[formula (3.19)]{sandier2023allen}, which is written in the following proposition within our specific context.   

\begin{proposition}[Energy upper bound]\label{prop: upper bound}
    There exist $C>0$ and $\alpha\in(0,1)$, both of which are independent of $\e$ and $R_0$, such that when $R_0$ is sufficiently large, the following energy upper bound holds:
    \begin{equation}\label{est: upper bound}
         \int_{B_{R_0}} \left( \f12|\na u|^2+W(u) \right)\,dz\leq \sigma(|DA|+|DB|+|DC|) +CR_0^\alpha. 
    \end{equation}
\end{proposition}

\begin{proof}

The proof follows essentially the proof of \cite[Proposition 3.3]{sandier2023allen} (Upper bound construction part), and therefore, it is omitted herein. We only briefly outline the core idea. 

To construct an energy competitor $\tu$, we first define a well-behaved boundary data on $\pa B_{R1}$ such that $\tu$ equals to $a_i$ on the rescaled arc $\f{R_1}{R_0}I_i$ for any $i=1,2,3$, where $R_1=R_0-R_0^{\alpha}$. On the annulus $A_{R_0,R_1}$, let $\tu$ linearly interpolate between $\tu|_{\pa B_{R_0}}$ and $\tu|_{\pa B_{R_1}}$ with the energy $E(\tu,A_{R_0,R_1})$ controlled by $C(R_0^\alpha+R_0^{1-\alpha})$. For the construction inside $B_{R_1}$, near the interface, say $DA$, we set $\tu(z)=U_{12}(d(z,DA))$ within a thin rectangle with a width of $R_0^\alpha$ and the long side parallel to $DA$. Here $d(\cdot)$ is the signed distance function. The energy within these rectangles will be approximately $\sigma(|DA|+|DB|+|DC|)$. These three thin rectangles partition $B_{R-1}$ into three subdomains, within which we simply take $\tu$ equals to phase $a_i$ corresponding to the boundary data. Again, interpolations are required in a $R^{\alpha}$-outer layer of the rectangles, with the energy being proved to be negligible. This construction parallels that of $v(x,y)$ in $B_{R_0-1}$ for the upper bound estimate of Case 2 in the proof of Lemma \ref{lemma: Yi nonempty}.

We also mention that a similar approach to construct the energy competitor also appears in \cite[Appendix A]{alikakos2024triple}, for the special case when $D$ coincides with the origin. 

\end{proof}

\begin{corol}\label{corol: D is close to origin}
    By selecting sufficiently small $\e$ and sufficiently large $R_0$ in Section \ref{sec: existence}, we can guarantee that the point $D$ locates in a small neighborhood of the origin, i.e. $D\in B_{\tau R_0}$, where $\tau=\min\{\frac{1}{100}, \frac{C_1}{20}\}$ for the constant $C_1$ in \eqref{meas est for transition arcs}. Furthermore, it follows that $D$ locates in the interior of the triangle $ABC$ and $\angle{ADB}=\angle{BDC}=\angle{ADC}=\f{2\pi}{3}.$
\end{corol}

\begin{rmk}
    We will explain the choice of $\tau$ in the proof of  Proposition \ref{prop: lower bound}.
\end{rmk}

\begin{proof}
    Assume by contradiction $D\not\in B_{\tau R_0}$, then an elementary calculation implies there exists a positive constant $\mu\sim O(\tau^2)$ such that 
    $$
    |DA|+|DB|+|DC|<(3-\mu) R_0.
    $$
    This estimate, together with \eqref{est: upper bound}. yields a contradiction with the lower bound in \eqref{est:ene in BR0} provided $R_0$ is sufficiently large. 
\end{proof}

Next we establish the lower bound for $E(u,B_{R_0})$. 

\begin{proposition}\label{prop: lower bound}
    There exists a constant $C(W)$ such that 
    \begin{equation}\label{est:lower bound}
         \int_{B_{R_0}} \left( \f12|\na u|^2+W(u) \right)\,dz\geq \sigma(|DA|+|DB|+|DC|)-CR_0^\f{2}{3}. 
    \end{equation}
\end{proposition}

\begin{proof}
 Up to a possible rotation, we assume $\overrightarrow{DA}$ is parallel to the positive $y$--axis, $\overrightarrow{DB}$ is parallel to the vector $(-\sqrt{3},-1)$ and $\overrightarrow{DC}$ is parallel to the vector $(\sqrt{3},-1)$, as shown in Figure \ref{fig lower bound}. We set
    \begin{equation*}
        A=(x_A,y_A),\ B=(x_B,y_B),\ C=(x_C,y_C),\ D=(x_D,y_D). 
    \end{equation*}
    By Corollary \eqref{corol: D is close to origin},
    \begin{equation*}
        |OD|=\sqrt{x_D^2+y_D^2}< \tau R_0, \quad \text{for }\tau=\min\{\frac{1}{100}, \frac{C_1}{20}\}.
    \end{equation*}
    Without loss of generality we assume that 
    \begin{equation*}
        x_D=x_A\geq 0, 
    \end{equation*}
    and therefore by simple geometry it holds that 
    \begin{equation*}
        y_C\geq y_B.
    \end{equation*}

    Now we define an extension $\tilde u$ of $u|_{B_{R_0}}$ to a larger ball $B_{R_0+1}$, which satisfies a simpler boundary condition on $\pa B_{R_0+1}$. We set $\tilde{I}_i$, $\tilde{I}_{ij}$ as the image of arcs $I_{i}$, $I_{ij}$ under the homothetic transformation: $z\ri \f{R_0+1}{R_0}z$, i.e. 
    \beqo
    \begin{split}
    \tilde{I}_i&:=\{z\in \pa B_{R_0+1}:\ \f{R_0}{R_0+1}z\in I_i\},\quad \forall i\in\{1,2,3\},\\
    \tilde{I}_{ij}&:=\{z\in \pa B_{R_0+1}:\ \f{R_0}{R_0+1}z\in I_{ij}\},\quad \forall i\neq j\in\{1,2,3\}.
    \end{split}
    \eeqo
   % Also set $\tilde{A}$, $\tilde{B}$, $\tilde{C}$, $\tilde{D}$ as the image of $A$, $B$, $C$, $D$ under the homothety. 
   The boundary value for $\tilde{u}$ on $\pa B_{R_0+1}$ is given by
    \begin{equation}\label{bdy value for lower bdd}
        \tilde{u}(z)=\begin{cases}
            a_i, & z\in \tilde{I}_i,\\
            \text{smooth connection of }a_i,\, a_j, & z\in \tilde{I}_{ij}.
        \end{cases}
    \end{equation}
We define $\tilde{u}$ in the annulus $A_{R_0+1,R_0}$ by linearly interpolating between $\tilde{u}|_{\pa B_{R_0+1}}$ and $u|_{\pa B_{R_0}}$, 
\begin{equation*}
    \tilde{u}(r,\theta)= (r-R_0)\tilde{u}(R_0+1,\theta)+(R_0+1-r)u(R_0,\theta),\ \ r\in (R_0,R_0+1).
\end{equation*}
Following the same computation as in the previous estimate of energy for the two phase boundary data (see Case 2 in the proof of Lemma \ref{lemma: Yi nonempty}), we can obtain the following upper bound for the energy in $A_{R_0+1,R_0}$:
\begin{equation*}
    \int_{A_{R_0+1,R_0}}\left( \f12|\na \tilde{u}|^2+W(\tu) \right)\,dz \leq C(\delta, W).
\end{equation*}
Consequently, in order to get \eqref{est:lower bound}, it suffices to prove
\begin{equation}\label{est: lower bdd on ext}
    \int_{B_{R_0+1}} \left( \f12|\na \tilde{u}|^2+W(\tu) \right)\,dz\geq \sigma(|DA|+|DB|+|DC|)-CR_0^\f{2}{3},
\end{equation}
with the boundary value \eqref{bdy value for lower bdd}. It follows from Corollary \ref{corol: D is close to origin} that $\tilde{u}$ equals to $a_1$, $a_2$, $a_3$ respectively on three arcs of $\pa B_{R_0+1}$ with the arc angle close to $\f{2\pi}{3}$. Moreover, these arcs are separated by small arcs of size $\f{C}{\delta^2}$ due to the estimate \eqref{meas est for transition arcs}.

Next we prove the lower bound \eqref{est: lower bdd on ext}. The argument essentially follows the proof of \cite[Proposition 3.1]{alikakos2024triple}, with necessary adjustments tailored to accommodate the scenario where $D$ is not the origin. Firstly, we extend $DA$ to intersect with $\pa B_{R_0+1}$ at $A'$. We claim that $A'\in \tilde{I}_{12}$. Taking $\tilde{A}$ to be the middle point of $\tilde{I}_{12}$ ($\tilde{A}$ is the intersection of the ray $\overrightarrow{OA}$ and $\pa B_{R_0+1}$). As illustrated in Figure \ref{fig A' in I12}, it follows from Corollary \ref{corol: D is close to origin} and \eqref{est d1 d2} that 
$$
|A'\tilde{A}| \leq |OD|\cdot \frac{2}{R_0}\leq 2\tau\leq \f{C_1}{10}\leq \f{3d_1}{10}. 
$$
which implies that $A'\in \tilde{I}_{12}$ since $\ch(\tilde{I}_{12})\geq d_1$. Similarly, we extend $DB$, $DC$ to intersect with $\pa B_{R_0+1}$ at $B'\in \tilde{I}_{13}$ and $C'\in \tilde{I}_{23}$ respectively. 

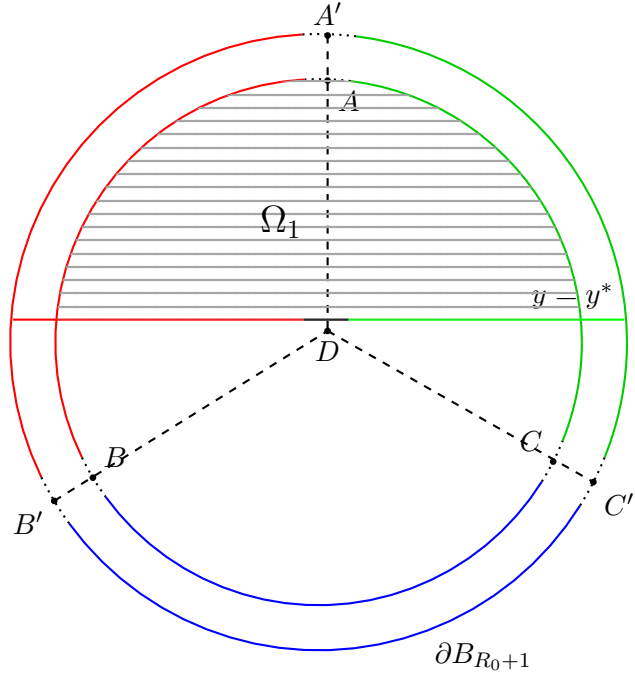
\begin{figure}[ht]
\centering
\begin{tikzpicture}[thick]
\draw [black!20!green,domain=0:83] plot ({3.5*cos(\x)}, {3.5*sin(\x)});
\draw [black!20!green,domain=338:360] plot ({3.5*cos(\x)}, {3.5*sin(\x)});
\draw [dotted,domain=83:93] plot ({3.5*cos(\x)}, {3.5*sin(\x)});
\draw [red,domain=93:206] plot ({3.5*cos(\x)}, {3.5*sin(\x)});
\draw [dotted,domain=206:216] plot ({3.5*cos(\x)}, {3.5*sin(\x)});
\draw [blue,domain=216:328] plot ({3.5*cos(\x)}, {3.5*sin(\x)});
\draw [dotted,domain=328:338] plot ({3.5*cos(\x)}, {3.5*sin(\x)});

\draw [black!20!green,domain=0:83] plot ({4.1*cos(\x)}, {4.1*sin(\x)});
\draw [black!20!green,domain=338:360] plot ({4.1*cos(\x)}, {4.1*sin(\x)});
\draw [dotted,domain=83:93] plot ({4.1*cos(\x)}, {4.1*sin(\x)});
\draw [red,domain=93:206] plot ({4.1*cos(\x)}, {4.1*sin(\x)});
\draw [dotted,domain=206:216] plot ({4.1*cos(\x)}, {4.1*sin(\x)});
\draw [blue,domain=216:328] plot ({4.1*cos(\x)}, {4.1*sin(\x)});
\draw [dotted,domain=328:338] plot ({4.1*cos(\x)}, {4.1*sin(\x)});

\draw ({4.1*cos(290)}, {4.1*sin(290)}) node [below right]{$\pa B_{R_0+1}$};

\draw [dashed] ({4.1*cos(211)},{4.1*sin(211)})--(0.12,0.15);
\draw[dashed] (0.12,0.15)--({4.1*cos(333)},{4.1*sin(333)});
\draw [dashed] (0.12,4.08)--(0.12,0.15);
\filldraw (0.12,0.15) circle (0.03cm) node[below]{$D$};
\filldraw (0.12,3.48) circle (0.03cm) node[below right]{$A$};
\filldraw ({3.5*cos(211)},{3.5*sin(211)}) circle (0.03cm) node[above right]{$B$};
\filldraw ({3.5*cos(333)},{3.5*sin(333)}) circle (0.03cm) node[above left]{$C$};

\filldraw (0.12,4.08) circle (0.03cm) node[above]{$A'$};
\filldraw ({4.1*cos(211)},{4.1*sin(211)}) circle (0.03cm) node[below left]{$B'$};
\filldraw ({4.1*cos(333)},{4.1*sin(333)}) circle (0.03cm) node[below right]{$C'$};

\draw [red] (-4.06,0.3)--(-0.2,0.3);
\draw (-0.2,0.3)--(0.4,0.3);
\draw [green] (0.4,0.3)--(4.06,0.3) node[above left,color=black] {$y=y^*$};

\path ({-3.5*cos(5)},{3.5*sin(5)}) coordinate (M);
\path ({3.5*cos(5)},{3.5*sin(5)}) coordinate (N);
\fill[pattern={Lines[angle=0,distance=5pt]}, pattern color=gray!70] (M)--(N) arc[start angle=5, end angle=175, radius=3.5] node at (-0.5,1.6) {\Large{$\Om_1$}};
 
\end{tikzpicture}
\caption{Triod $T_{ABC}$ with the center $D$ close to origin, and its extension to $T_{A'B'C'}\in B_{R_0+1}$}.
\label{fig lower bound}
\end{figure}

\begin{figure}[ht]
\centering
\begin{tikzpicture}[thick]
\draw [black!20!green,domain=20:74] plot ({4.2*cos(\x)}, {4.2*sin(\x)});
\draw [red,domain=90:160] plot ({4.2*cos(\x)}, {4.2*sin(\x)});
\draw [dotted,domain=74:90] plot ({4.2*cos(\x)}, {4.2*sin(\x)});

\draw [black!20!green,domain=20:74] plot ({5.5*cos(\x)}, {5.5*sin(\x)});
\draw [red,domain=90:160] plot ({5.5*cos(\x)}, {5.5*sin(\x)});
\draw [dotted,domain=74:90] plot ({5.5*cos(\x)}, {5.5*sin(\x)});

\filldraw (0,0) circle (0.05cm) node[below]{$O$};
\filldraw (0.6,0.2) circle (0.05cm) node[below]{$D$};
\filldraw ({4.2*cos(82)},{4.2*sin(82)}) circle (0.05cm) node[below right]{$A$};
\filldraw ({5.5*cos(82)},{5.5*sin(82)}) circle (0.05cm) node[below right]{$\tilde{A}$};

\draw[dashed](0,0)--({5.5*cos(82)},{5.5*sin(82)});
\draw[dashed](0.6,0.2)--({4.2*cos(82)},{4.2*sin(82)})--(0.6,5.46);
\filldraw (0.6,5.46) circle (0.05cm) node[below left]{$A'$};

\end{tikzpicture}
\caption{$A'$ locates within $\tilde{I}_{12}$.}
\label{fig A' in I12}
\end{figure}
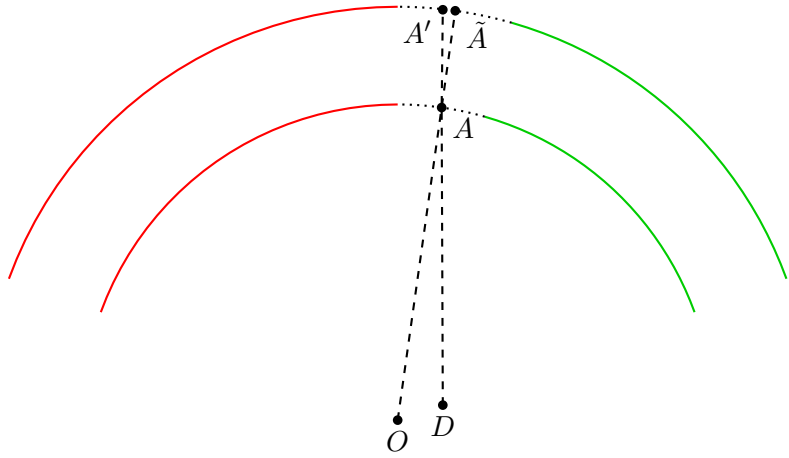

The coordinates for $A'$, $B'$, $C'$ are denoted as
\begin{equation*}
        A'=(x_{A}',y_A'),\quad B'=(x_{B}',y_{B}'),\quad C'=(x_C',y_C'),
\end{equation*}
and satisfy
\begin{equation*}
\begin{split}
   &x_A'=x_D=x_A,\quad y_C'\geq y_B',\\
   &\frac{y_B'-y_D}{x_B'-x_D}=\f{1}{\sqrt{3}},\quad \frac{y_C'-y_D}{x_C'-x_D}=-\f{1}{\sqrt{3}}.
\end{split}
\end{equation*}

For any $y\in(y_C',R_0+1)$, we define the horizontal line segment $\g_y$ and functions $\lam_1(y)$, $\lam_2(y)$ and $\lam_3(y)$ as follows. 
\begin{align*}
    \g_y&:=\{(x,y):\; x\in\mathbb{R}\}\cap B_{R_0+1},\\
    \lam_i(y)&:=\mathcal{H}^1(\g_y\cap \{\vert \tilde{u}(x,y)-a_i\vert \leq R_0^{-\f16}\}),\quad i\in \{1,2,3\}.
\end{align*}

By the boundary condition on $\pa B_{R_0+1}$ we know for any $y\in [y_C'+d_2, y_A'-d_2]$, it holds that $\lam_1(y)>0$ and $\lam_2(y)>0$. 

We set
\begin{align}
\label{def: y*} y^*&:=\min\{y\in [y_C'+d_2,R_0+1]:\lam_1(y)+\lam_2(y)\geq \ch(\g_y)-R_0^{\f23}\},\\
\nonumber \zeta(x)&:=\min\{y^*,\sqrt{(R_0+1)^2-x^2}\},\\
\nonumber K&:=\{x\in[x_B'+d_2, x_C'-d_2]:\  \vert \tu(x,\zeta(x))-a_i\vert <R_0^{-\f16},\ i=1 \text{ or } 2\},\\
\nonumber L&:=\{y\in [y_C'+d_2,y^*]: \lam_3(y)> 0\},\\
\nonumber \Om_1&:=\{z=(x,y)\in B_{R_0+1}: \ y\geq y^*\},\\
\nonumber \Om_2&:=\{z=(x,y)\in B_{R_0+1}: \ y< y^*\}.
\end{align}

Note that when $(R_0+1)^2-y^2<\f{R_0^{\f43}}{4}$, $\mathcal{H}^1(\g_y)-R_0^{\f23}<0$, and therefore the set on which we minimize to get $y^*$ is not empty. We will examine two scenarios based on the value of $y^*$. Additionally, we can bound the measure of $K$ from below by the definition of $y^*$ and the boundary constraint $\tu|_{\pa B_{R_0+1}}$:
\begin{equation}\label{est: meas of K}
    \ch(K)\geq x_C'-x_B'-2d_2-R_0^{\f23}.
\end{equation}

\begin{case}  $y*< y_A'-d_2$.

   For any $y\in[y^*,y_A'-d_2)$, the horizontal line $\g_y$ intersects $\pa B_{R_0+1}$ at two points, where $\tilde{u}$ takes the value $a_1$, $a_2$ respectively. The one dimensional energy estimate Lemma \ref{lemma: 1D energy estimate} yields
   \begin{equation*}
       \int_{\g_y} \left( \f12|\pa_x \tu|^2+W(\tu) \right)\,dx\geq \sigma.
   \end{equation*}
   Integrating along the vertical direction implies
   \begin{equation}\label{est: ene in Omega_1}
       \int_{\Om_1} \left( \f12|\pa_x \tu|^2+W(\tu) \right)\,dz\geq \sigma(y_A'-y^*-d_2)\geq \sigma(y_A'-y^*)-C(\delta).
   \end{equation}

On the domain $\Om_2$, we claim that there exists a constant $C$ such that 
\begin{equation*}
    \mathcal{H}^1([y_C'+d_2,y^*]\setminus L)<CR_0^{\f23}.
\end{equation*}

Indeed, we set 
\beqo
S:=\{y\in[y_C'+d_2,y^*]: \lam_3(y)=0\}=[y_C'+d_2,y^*]\setminus L.
\eeqo
For $y\in S$, the definitions of $y^*$ and $S$ imply that $\lam_1(y)+\lam_2(y)+\lam_3(y)<\ch(\g_y)-R_0^{\f23}$, or equivalently
\begin{equation*}
    \ch(\{x \in [-\sqrt{(R-_0+1)^2-y^2},\sqrt{(R_0+1)^2-y^2}]: \vert \tu(x,y)-a_i\vert >R_0^{-\f16},\ \forall i\})>R_0^{\f23}.
\end{equation*}
From the energy upper bound \eqref{est: upper bound} and Lemma \ref{lemma: potential energy estimate} we have that
\beqo
\begin{split}
    &4R_0\geq \int_S\int_{\g_y} W(\tu)\ dxdy\geq  \ch(S)\cdot(\f12c_W R_0^{-\f13})\cdot R_0^{\f23}\\
    &\Rightarrow \ch(S)\leq CR_0^{\f23},
\end{split}
\eeqo
for some constant $C$ depending on $W$. The claim is established.

Next we want to estimate the energy in $\Om_2$ in both vertical and horizontal directions. We split the potential $W(\tilde{u})$ as 
\begin{equation*}
    W(\tu)=\sin^2\t \ W(\tu)+\cos^2\t\  W(\tu),
\end{equation*}
for some $\t\in[0,\frac{\pi}{2}]$ to be determined. 

For $x\in K$, $A(x)=(x,\zeta(x))$ and $B(x)=(x,-\sqrt{(R_0+1)^2-x^2})$ are the two intersection points of the vertical line $\{(x,y):y\in\BR\}$ with $\pa\Om_2$. We have $\tu(B(x))=a_3$ by the boundary condition and $|\tu(A(x))-a_i|\leq R_0^{-\f16}$ for $i=1$ or $2$ by the definition of $K$. Applying Lemma \ref{lemma: 1D energy estimate} and then integrating with respect to $x$ yields
\begin{equation}\label{ene: om2 vertical}
\begin{split}
    &\int_{\Om_2}\left( \f{1}2|\pa_y \tu|^2+\sin^2\theta\  W(\tu) \right)\,dz\\
    \geq & \sin\t\int_{x\in K} \int_{-\sqrt{(R_0+1)^2-x^2}}^{\zeta(x)}\left(\f{1}{2\sin\t}|\pa_y \tu|^2+\sin\theta\,W(\tu) \right)\,dydx\\
    \geq & \sin\t(x_C'-x_B'-2d_2-R_0^{\f23}) (\sigma-CR_0^{-\f13})\\
    \geq & \sin\t\left[(x_C'-x_B')\sigma-CR_0^\f{2}{3}\right].  
\end{split}
\end{equation}
Here we utilize $x_C'-x_B'\sim \sqrt{3}R_0$ and $d_2\leq \f{C}{\delta^2}\ll R_0^{\f23}$ to derive the last inequality.

By definition, for any $y\in L$, 
\begin{align*}
    &\tu(-\sqrt{(R_0+1)^2-y^2},y)=a_1,\quad \tu(\sqrt{(R_0+1)^2-y^2},y)=a_2,\\
    \exists &\, (x_0,y)\in \g_y,\ \text{ such that }|\tu(x_0,y)-a_3|\leq R_0^{-\f16},
\end{align*}
which implies that there are approximately two phase transitions along $\g_y$. Thus we can estimate 
\begin{equation}\label{ene: om2 horizontal 1}
    \begin{split}
     &\int_{\Om_2\cap\{y\in L\}} \left( \f12\vert\pa_{x}\tu\vert^2+\cos^2\theta\,W(\tu) \right)\,dz\\
     \geq & \cos\t\,\int_{y\in L} \int_{\g_y}  \left(  \f1{2\cos\t}\vert\pa_{x}\tu\vert^2+\cos\theta\,W(\tu) \right)\,dxdy\\
    =& \cos\theta\,\int_{y\in L} \left\{ \int_{-\sqrt{(R_0+1)^2-y^2}}^{x_0} +\int_{x_0}^{\sqrt{(R_0+1)^2-y^2}}\right\} \left(\f{1}{2\cos\theta_2}\vert\pa_{x}\tu\vert^2+ \cos\theta\,W(u)\right)\,dx\\
    \geq & \cos\theta\cdot(y^*-y_C'-d_2-CR_0^{\f23})\cdot(2\sigma-CR_0^{-\f13})\\
    \geq & 2\cos\t (y^*-y_C')\sigma- CR_0^{\f23},
     \end{split}
\end{equation}
where the last inequality follows from $y^*-Y_C'<2R_0$ and $d_2\ll R_0^{\f23}$. 

In case $y_B'+d_2\geq y_C'-d_2$, we directly proceed with the above estimate \eqref{ene: om2 horizontal 1} to derive \eqref{ene: om2 horizontal}. Otherwise, if $y_B'+d_2 < y_C'-d_2$, for any $y\in (y_B'+d_2,y_C'-d_2)$, $\g_y$ will intersect with $\pa B_{R_0+1}$ at two points where $\tu$ equals to $a_1,\,a_3$ respectively. We have that 
\begin{equation}\label{ene: om2 horizontal 2}
    \begin{split}
     &\int_{\Om_2\setminus\{y\in L\}} \left( \f12\vert\pa_{x}\tu\vert^2+\cos^2\theta\,W(\tu) \right)\,dxdy\\
     \geq & \cos\t\,\int_{y_B'+d_2}^{y_C'-d_2} \int_{\g_y}  \left(  \f1{2\cos\t}\vert\pa_{x}\tu\vert^2+\cos\theta\,W(\tu) \right)\,dxdy\\
    =& \cos\theta\,(y_C'-y_B'-2d_2)\sigma\\
    \geq & \cos\t\,(y_C'-y_B')\sigma- C.
     \end{split}
\end{equation}

Adding \eqref{ene: om2 horizontal 1} and \eqref{ene: om2 horizontal 2} gives
\begin{equation}\label{ene: om2 horizontal} 
    \int_{\Om_2} \left( \f12\vert\pa_{x}\tu\vert^2+\cos^2\theta\,W(\tu) \right)\,dz\geq \cos\t\,(2y^*-y_B'-y_C')-CR_0^{\f23}.
\end{equation}
This, together with \eqref{ene: om2 vertical}, implies that 
\begin{equation*}
    \int_{\Om_2}\left( \f12|\na \tu|^2+W(\tu) \right)\,dz
    \geq  \left[\sin\t\,(x_C'-x_B')+\cos\t(2y^*-y_B'-y_C')\right]\sigma-CR_0^{\f23},
\end{equation*}
holds for any $\theta\in[0,\f{\pi}{2}]$. Taking $\t=\arctan{\f{x_C'-x_B'}{2y^*-y_B'-y_C'}}$ to maximize the right-hand side, we obtain
\begin{equation}\label{ene: om2}
    \int_{\Om_2} \left( \f12|\na \tu|^2+W(\tu) \right)\,dz
    \geq \sigma\cdot \sqrt{(x_C'-x_B')^2+(2y^*-y_B'-y_C')^2} -CR_0^{\f23}.
\end{equation}

Combining \eqref{est: ene in Omega_1} and \eqref{ene: om2} gives that 

\begin{equation}\label{lower bound by y*}
    \begin{split}
    \int_{B_{R_0+1}}& \left( \f12|\na \tu|^2+W(\tu) \right)\,dz\\
    \geq & \sigma\left[ (y_A'-y^*)+ \sqrt{(x_C'-x_B')^2+(2y^*-y_B'-y_C')^2}\right] -CR_0^{\f23}.
    \end{split}
\end{equation}

We are left with solving the following minimization problem
\begin{equation}\label{minimization problem: y*}
    \begin{split}
        \min & \quad  (y_A'-y^*)+ \sqrt{(x_C'-x_B')^2+(2y^*-y_B'-y_C')^2},\\
        \text{subject to} &\quad  y^*\in(y_C'+d_2,y_A'-d_2).
    \end{split}
\end{equation}
Direct calculation shows that 
\begin{equation*}
    \min  \ \left\{(y_A'-y^*)+ \sqrt{(x_C'-x_B')^2+(2y^*-y_B'-y_C')^2}\right\}=|DA'|+|DB'|+|DC'|,
\end{equation*}
and the minimum is reached at 
\begin{equation*}
    y^*=y_D.
\end{equation*}

The calculation is elementary and provided in Appendix \ref{minimization y*}. Therefore \eqref{est:lower bound} holds for the case $y^*<y_A'-d_2$.

\end{case}

\begin{case} $y^*\geq y_A'-d_2$. We show that in this case the energy is strictly larger than $\sigma(|DA'|+|DB'|+|DC'|)$. Split
\begin{equation*}
W(\tu)=\sin^2\t\,W(\tu)+\cos^2\t\,W(\tu),\text{ for some }\t\in(0,\f{\pi}{2}).
\end{equation*}

By the boundary data we have
\begin{align*}
    &\tu(x,\sqrt{(R_0+1)^2-x^2})=a_1,\,\tu(x,-\sqrt{(R_0+1)^2-x^2})=a_3,\ \forall x\in [x_{B'}+d_2,x_A'-d_2],\\
    &\tu(x,\sqrt{(R_0+1)^2-x^2})=a_2,\,\tu(x,-\sqrt{(R_0+1)^2-x^2})=a_3,\ \forall x\in [x_A'+d_2,x_C'-d_2],
\end{align*}
which allows us to estimate the energy in the vertical direction,
\begin{equation}\label{ene:case2,vertical}
    \begin{split}
&\int_{B_{R_0+1}}\left(\f1{2}\vert\pa_y \tu\vert^2+\sin^2\theta\,W(\tu)\right)\,dz\\
        \geq & \int_{x_{B'}+d_2}^{x_A'-d_2} \sigma \sin\theta\,dx+\int_{x_{A'}+d_2}^{x_C'-d_2} \sigma\sin\theta\,dx\\
        \geq & \s\sin\theta\,(x_C'-x_B')-C.  
    \end{split}
\end{equation}

Applying the same argument as in the claim of Case 1, we have that most of $y\in [y_C'+d_2,y_A'-d_2]$ belong to $L$, more precisely,
\begin{equation*}
    \ch([y_C'+d_2,y_A'-d_2]\setminus L)<CR_0^{\f23}.
\end{equation*}

For $y\in L\cap[y_C'+d_2,y_A'-d_2]$, by definition there are two phase transitions: from $a_1$ to $a_3$ and from $a_3$ to $a_2$. We perform a similar computation as in \eqref{ene: om2 horizontal 1} to obtain
\begin{equation}\label{ene: case 2 horizontal}
    \begin{split}
     &\int_{B_{R_0+1}\cap\{y\in L\}} \left( \f12\vert\pa_{x}\tu\vert^2+\cos^2\theta\,W(\tu) \right)\,dxdy\\
     \geq & \cos\t\,\int_{y\in L\cap [y_C'+d_2,y_A'-d_2]} \int_{\g_y}  \left(  \f1{2\cos\t}\vert\pa_{x}\tu\vert^2+\cos\theta\,W(\tu) \right)\,dxdy\\
    \geq & 2\cos\t (y_A'-y_C')\sigma- CR_0^{\f23}.
     \end{split}
\end{equation}

Adding \eqref{ene:case2,vertical} and \eqref{ene: case 2 horizontal} and maximizing with respect to $\t$ yield
\begin{equation}\label{ene: case2}
\begin{split}
    \int_{B_{R_0+1}} \left( \f12|\na \tu|^2+W(\tu) \right)\,dz &\geq \sigma\,\sqrt{(x_C'-x_B')^2+4(y_A'-y_C')^2}-CR_0^{\f23}\\
    &\geq \f{49}{50} 2\sqrt{3} R_0 \sigma. 
\end{split}
\end{equation}
where the last line follows from Corollary \ref{corol: D is close to origin} which implies that $x_C'-x_B'\geq (1-2\tau)\sqrt{3}R_0$ and $y_A'-y_C'\geq  (1-2\tau)\f{3R_0}{2}$. Note that 
\beqo
|DA'|+|DB'|+|DC'|\leq 3R_0+3,
\eeqo
which together with \eqref{ene: case2} leads to
\beqo
 \int_{B_{R_0+1}} \left( \f12|\na \tu|^2+W(\tu) \right)\,dz >\sigma(|DA'|+|DB'|+|DC'|).
\eeqo
This completes the proof of Proposition \ref{prop: lower bound}.
\end{case}

\end{proof}

At the end of this section, we will derive that $y^*$ is not far from $y_D$. Set 
\begin{equation}\label{def of beta}
    \beta:=1- \min\left\{\f16,\f{1-\al}{2}\right\}\in (0,1),
\end{equation}
where $\al\in(0,1)$ is the constant in Proposition \ref{prop: upper bound}. Then we have the following lemma:
\begin{lemma}\label{lemma: est of y*}
   There is a constant $C$ depending on $W$, such that 
   \beqo
   |y^*-y_D|\leq CR_0^{\beta},
   \eeqo
   where $y^*$ is defined in \eqref{def: y*}.
\end{lemma}

\begin{proof}
By examining the proof of Proposition \ref{prop: lower bound}, we have 
\beqo
\int_{B_{R_0+1}}\left(\f12|\na\tu|^2+W(\tu)\right)\,dz\leq \int_{B_{R_0}}\left( \f12|\na\tu|^2+W(\tu) \right)\,dz+C(\delta,W).
\eeqo

This together with the upper bound \eqref{est: upper bound} and \eqref{lower bound by y*} implies

\begin{equation}
\begin{split}
\sigma\bigg[ (y_A'-y^*)+ &\sqrt{(x_C'-x_B')^2+(2y^*-y_B'-y_C')^2}\bigg]\\
&\leq \sigma(|DA|+|DB|+|DC|)+CR_0^\al+CR_0^{\f23}.
\end{split}
\end{equation}

From Appendix \ref{minimization y*}, the left-hand side attains its minimum value $\sigma(|DA|+|DB|+|DC|)$ when $y^*=y_D$. When the functional is perturbed by a small quantity $R_0^\al+R_0^{\f23}$, it is straightforward to show that $y^*$ can be perturbed away from $y_D$ by an amount no greater than $CR_0^{\beta}$, where $\beta$ is defined by \eqref{def of beta}.

\end{proof}

\section{Localization of the diffuse interface}\label{sec: loc of interface}
    
We insist on the choice of $A,B,C,D$ at the beginning of Section \ref{sec: lower upper bdd}, and assume without loss of generality that $\overrightarrow{DA}$ is parallel to the positive $y$-axis. Then the triod defined by 
$$
T_{ABC}=DA\cup DB\cup DC
$$
divides $B_{R_0}$ into three regions:
\begin{align*}
    \mathcal{D}_1 &:=\text{ the region enclosed by }DA,\,DB,\,\pa B_{R_0};\\
    \mathcal{D}_2 &:=\text{ the region enclosed by }DA,\,DC,\,\pa B_{R_0};\\
    \mathcal{D}_3 &:=\text{ the region enclosed by }DB,\,DC,\,\pa B_{R_0}.    
\end{align*}

For any $\g>0$, we define the diffuse interface $\mathcal{I}_\g$ by
\begin{equation*}
    \mathcal{I}_\g:= \{z: |u(z)-a_i|\geq \g,\ \forall\, i=1,2,3\}.
\end{equation*}

From now on we fix $\g=2C\delta$, where $C=C(W)$ is the constant in Lemma \ref{lemma: on bdy, close to each phase} which measures the closeness of $u$ to $a_i$ on arc $I_i$. Note that $C$ is independent of $\delta$, which allows $\g$ to be arbitrarily small by choosing suitably small $\delta$. The main result of this section is the following proposition which states that $\mathcal{I}_{\g}$ is contained in a $O(R_0^\beta)$ neighborhood of $T_{ABC}$. 

\begin{proposition}\label{prop: loc of interface}
 There exists a constant $C=C(\g,W)$, such that for sufficiently small $\e$ and associated $R_0$, it holds that
\beq\label{diffused interface close to T}
 \quad \mathcal{I}_{\g}\cap B_{R_0}\subset N_{CR_0^{\beta}}(T_{ABC}):=\{z\in B_{R_0}: \dist(z,T_{ABC})\leq CR_0^\beta\}.
\eeq
Moreover, there are positive constants $K=K(M,W)$ and $k=k(W)$ such that
\beq\label{distance to a_i}
\vert u(z)-a_i\vert \leq Ke^{-{k}(\dist(z,\pa \mathcal{D}_i)-CR_0^\beta)^+},\quad z\in \mathcal{D}_i,\ i=1,2,3,
\eeq
where $(a)^+=\max\{a,0\}$.  
\end{proposition}

\begin{proof}
    Let $C_0$ be the constant in Lemma \ref{lemma: est of y*} and $\tu$ be defined as the extension of $u|_{B_{R_0}}$ onto $B_{R_0+1}$ in the proof of Proposition \ref{prop: lower bound}. 

    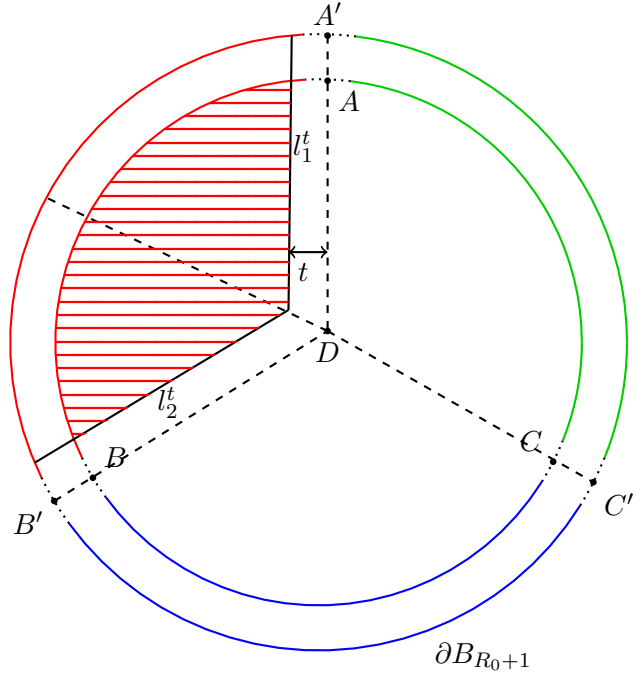
\begin{figure}[ht]
\centering
\begin{tikzpicture}[thick]
\draw [black!20!green,domain=0:83] plot ({3.5*cos(\x)}, {3.5*sin(\x)});
\draw [black!20!green,domain=338:360] plot ({3.5*cos(\x)}, {3.5*sin(\x)});
\draw [dotted,domain=83:93] plot ({3.5*cos(\x)}, {3.5*sin(\x)});
\draw [red,domain=93:206] plot ({3.5*cos(\x)}, {3.5*sin(\x)});
\draw [dotted,domain=206:216] plot ({3.5*cos(\x)}, {3.5*sin(\x)});
\draw [blue,domain=216:328] plot ({3.5*cos(\x)}, {3.5*sin(\x)});
\draw [dotted,domain=328:338] plot ({3.5*cos(\x)}, {3.5*sin(\x)});

\draw [black!20!green,domain=0:83] plot ({4.1*cos(\x)}, {4.1*sin(\x)});
\draw [black!20!green,domain=338:360] plot ({4.1*cos(\x)}, {4.1*sin(\x)});
\draw [dotted,domain=83:93] plot ({4.1*cos(\x)}, {4.1*sin(\x)});
\draw [red,domain=93:206] plot ({4.1*cos(\x)}, {4.1*sin(\x)});
\draw [dotted,domain=206:216] plot ({4.1*cos(\x)}, {4.1*sin(\x)});
\draw [blue,domain=216:328] plot ({4.1*cos(\x)}, {4.1*sin(\x)});
\draw [dotted,domain=328:338] plot ({4.1*cos(\x)}, {4.1*sin(\x)});

\draw ({4.1*cos(290)}, {4.1*sin(290)}) node [below right]{$\pa B_{R_0+1}$};

\draw [dashed] ({4.1*cos(211)},{4.1*sin(211)})--(0.12,0.15);
\draw[dashed] (0.12,0.15)--({4.1*cos(333)},{4.1*sin(333)});
\draw[dashed] (0.12,0.15)--({4.1*cos(152)},{4.1*sin(152)});
\draw [dashed] (0.12,4.08)--(0.12,0.15);
\filldraw (0.12,0.15) circle (0.03cm) node[below]{$D$};
\filldraw (0.12,3.48) circle (0.03cm) node[below right]{$A$};
\filldraw ({3.5*cos(211)},{3.5*sin(211)}) circle (0.03cm) node[above right]{$B$};
\filldraw ({3.5*cos(333)},{3.5*sin(333)}) circle (0.03cm) node[above left]{$C$};

\filldraw (0.12,4.08) circle (0.03cm) node[above]{$A'$};
\filldraw ({4.1*cos(211)},{4.1*sin(211)}) circle (0.03cm) node[below left]{$B'$};
\filldraw ({4.1*cos(333)},{4.1*sin(333)}) circle (0.03cm) node[below right]{$C'$};

\path ({3.5*cos(202)},{3.5*sin(202)}) coordinate (M);
\path ({3.5*cos(96)},{3.5*sin(96)}) coordinate (N);
\path (-0.4,0.43) coordinate (Q);
\draw ({4.1*cos(203)},{4.1*sin(203)})--(Q) node at (-2,-0.8) {$l_2^t$};
\draw ({4.1*cos(95)},{4.1*sin(95)})--(Q) node at (-0.2,2.6) {$l_1^t$};
\fill[pattern={Lines[angle=0,distance=5pt]}, pattern color=red] (M)--(Q)--(N) arc(96:202:3.5);
\draw [<->] (-0.4,1.2)--(0.12,1.2) node at (-0.2,1.2) [below] {$t$}; 
 
\end{tikzpicture}
\caption{Definition of $l_1^t$, $l_2^t$. In the red shaded region, $|u(z)-a_1|\leq \g$}.
\label{fig l1t l2t}
\end{figure}

For any $t\in [0,\f12R_0]$, we define the line segments (see Figure \ref{fig l1t l2t})
\begin{equation}
\begin{split}
    &l_1^t:=\{(x_D-t,y):\ \f{y-y_D}{t}\geq \f{\sqrt{3}}{3}\}\cap \overline{B}_{R_0+1},\\
    &l_2^t:=\{(x,y):\ \f{y-(y_D+\f{\sqrt{3}}{3}t)}{x-(x_D-t)}=\f{\sqrt3}{3},\; x\leq x_D-t\}\cap \overline{B}_{R_0+1}. 
\end{split}
\end{equation}
From Lemma \eqref{lemma: est of y*},
when $t\geq \sqrt3 C_0 R_0^{\beta}$, we have that $l_1^t \subset \overline{\Om}_1$, for $\Om_1$ defined by $B_{R_0+1}\cap \{y\geq y^*\}$.

We set
\begin{equation*}
    \mathcal{A}:=\{t: t\in [\sqrt3 C_0 R_0^{\beta},\f12 R_0], \; \max\limits_{z\in l_1^t} \vert \tu(z)-a_1\vert >\g\}.
\end{equation*}
If $\mathcal{A}=\varnothing$ then we trivially get  \eqref{est of H^1(A)}. Now suppose $\mathcal{A}\neq \varnothing$, we show the measure of $\mathcal{A}$ is $O(R_0^{\beta})$. The proof relies on the part of the lower bound over $\Om_1$ derived in \eqref{est: ene in Omega_1} only considers the horizontal gradient, thereby allowing for the addition of vertical displacement to provide improvement.

For any $t\in \mathcal{A}$, there exists a point $z_t\in l_1^t$ such that 
\beqo
\vert \tu(z_t)-a_1\vert >\g.
\eeqo
By the boundary data on $\pa B_{R_0+1}$, 
\beqo
\tu(x_D-t,\sqrt{(R_0+1)^2-(x_D-t)^2})= a_1. 
\eeqo
Thus there exists a constant $C_1:=C_1(\g,W)$ such that 
\beq\label{energy on l_1^t}
\int_{y_D+\f{\sqrt3t}{3}}^{\sqrt{(R_0+1)^2-(x_D-t)^2}} \left(\f{\lambda}{2}\vert \pa_y \tu\vert ^2+\f{1}{\lambda}W(\tu)\right)\,dy\geq C_1, \quad \forall \lambda>0.
\eeq
Set 
\beqo
\kappa:=\f{\mathcal{H}^1(\mathcal{A})C_1}{\s(y_A'-y^*)}.
\eeqo
We compute the energy on $\Om_1$:
\begin{equation}\label{est on Om1: new}
    \begin{split}
        &\int_{\Om_1} \left(\f{1}{2}\vert \na \tu\vert ^2+W(\tu)\right)\,dxdy\\
        = & \int_{\Om_1} \left( \f{1}{2}\vert \pa_x \tu\vert ^2+\f{1}{1+\kappa^2}W(\tu) \right)\,dxdy\\
        &\qquad +\int_{\Om_1}\left( \f{1}{2}\vert \pa_y \tu\vert ^2+\f{\kappa^2}{1+\kappa^2} W(\tu) \right)\,dydx\\
        \geq & \f{1}{\sqrt{1+\kappa^2}} \int_{\Om_1} \left(\f{\sqrt{1+\kappa^2}}{2} \vert \pa_x \tu\vert ^2 +\f{1}{\sqrt{1+\kappa^2}}W(\tu)  \right)\,dxdy\\
        &\qquad  + \f{\kappa}{\sqrt{1+\kappa^2}} \int_{\Om_1}\left(\f{\sqrt{1+\kappa^2}}{2\kappa} \vert \pa_y \tu\vert ^2+\f{\kappa}{\sqrt{1+\kappa^2}} W(\tu)  \right)\,dydx\\
        \geq & \f{1}{\sqrt{1+\kappa^2}}\sigma(y_A'-y^*)\\
        &\qquad +\f{\kappa}{\sqrt{1+\kappa^2}}\int_\mathcal{A}\int_{l_1^t} \left(\f{\sqrt{1+\kappa^2}}{2\kappa} \vert \pa_y \tu\vert ^2+\f{\kappa}{\sqrt{1+\kappa^2}} W(\tu)  \right)\,dydx\\
        \geq & \f{1}{\sqrt{1+\kappa^2}} \sigma (y_A'-y^*)+\f{\kappa}{\sqrt{1+\kappa^2}}\mathcal{H}^1(\mathcal{A})C_1\\
        =& \sqrt{\sigma^2(y_A'-y^*)^2+(C_1\mathcal{H}^1(\mathcal{A}))^2}.
    \end{split}
\end{equation}

This, together with \eqref{ene: om2}, updates the total energy in $B_{R_0+1}$:
\beqo
\begin{split}
    & \int_{B_{R_0+1}} \left( \f{1}{2}\vert \na \tu\vert ^2+ W(\tu)  \right)\,dz\\
    \geq & \sqrt{\sigma^2(y_A'-y^*)^2+(C_1\mathcal{H}^1(\mathcal{A}))^2}+\sigma\sqrt{(x_C'-x_B')^2+(2y^*-y_B'-y_C')^2}-CR_0^{\f23} \\
    \geq& \sigma\bigg[ (y_A'-y^*)+ \sqrt{(x_C'-x_B')^2+(2y^*-y_B'-y_C')^2}\bigg]-CR_0^{\f23}+\f{(C_1\mathcal{H}^1(\mathcal{A}))^2}{2\sigma(y_A'-y^*)}\\
    \geq & \sigma(|DA'|+|DB'|+|DC'|)-CR_0^{\f23}+\f{(C_1\mathcal{H}^1(\mathcal{A}))^2}{2\sigma(y_A'-y^*)}.
\end{split}
\eeqo
Combining this with the upper bound \eqref{est: upper bound} implies 
\beq\label{est of H^1(A)}
\f{(C_1\mathcal{H}^1(\mathcal{A}))^2}{2\sigma (y_A'-y^*)}\leq C(R_0^{\f23}+R_0^\al)\; \Rightarrow \; \mathcal{H}^1(\mathcal{A})\leq C_2(\g,W) R_0^{\beta}. 
\eeq

We define
\beq\label{def of B}
\mathcal{B}:=\{t: t\in [\sqrt3C_0R_0^{\beta},\f12R_0], \; \max\limits_{z\in l_2^t} \vert u(z)-a_1\vert>\g\}.
\eeq
An analogous computation implies 
\beq\label{est of H^1(B)}
\mathcal{H}^1(\mathcal{B})\leq C_2(\g,W)R_0^{\beta}. 
\eeq

Set
\beqo
C_3(\g,W):=\sqrt3 C_0+3C_2.
\eeqo
From \eqref{est of H^1(A)} and \eqref{est of H^1(B)} it follows that  there exists $t_0\in [\sqrt3C_0R_0^{\beta}, C_3R_0^{\beta}]$ (note that $R_0$ is chosen to be large enough so that $C_3R_0^{\beta}<\f12R_0$) such that 
\beqo
 \vert \tu(z)-a_1\vert \leq \g,\quad \forall z\in l_1^{t_0}\cup l_2^{t_0}.
\eeqo

Let $\mathcal{D}_{1,\g}\subset \mathcal{D}_1$ denote the region enclosed by $l_1^{t_0}$, $l_2^{t_0}$ and $\pa B_{R_0}$. By Lemma \ref{lemma: on bdy, close to each phase}, we have
\beqo
\vert u(z)-a_1\vert \leq \g,\quad \forall z\in \pa \mathcal{D}_{1,\g}.
\eeqo
Here we write $u$ instead of $\tu$ because $u=\tu$ on $\overline{B}_{R_0}$. 

By the variational maximum principle (see Lemma \ref{lemma: maximum principle}), given $\g$ (or equivalently, $\delta$) small enough, we conclude that
\begin{equation}\label{maximum principle}
    \vert u(z)-a_1\vert \leq \g\ \text{ on }\mathcal{D}_{1,\g}.
\end{equation}
This further implies that the diffuse interface $\mathcal{I}_{\g}\cap \mathcal{D}_1$ is contained in a $C_3R_0^{\b}$ neighborhood of $T$. The same argument works for $\mathcal{I}_{\g}\cap \mathcal{D}_j$ for $j=2,3$ and we conclude the proof of the first part \eqref{diffused interface close to T}.

Utilizing \eqref{maximum principle}, hypothesis (H1) and the comparison principle for elliptic equations, the exponential decay estimate \eqref{distance to a_i} follows (cf. \cite[Proposition 6.4]{afs-book}), which completes the proof. 

\end{proof}

\section{Rescaling to the unit disk $B_1$}\label{sec: rescale}

In the following two sections, we initially choose a small $\delta$, ensuring that the specified smallness criteria for $\delta$ in the proofs of Lemma \ref{lemma: on bdy, close to each phase} and Proposition \ref{prop: loc of interface} are met. The choice of $\delta$ depends only on the potential $W$. Subsequently we determine a possibly even smaller $\e$ and the associated $R_0(\e)\in (R(\e),2R(\e))$ (see Section \ref{sec: existence} for the selection of $R_0(\e)$), ensuring all energy bounds established in Section \ref{sec: existence}, Section \ref{sec: lower upper bdd} and Section \ref{sec: loc of interface} are satisfied, and also fulfilling the requirements on the largeness of $R_0$ in the proofs of Lemma \ref{lemma: Yi nonempty} and Lemma \ref{lemma: on bdy, close to each phase}. Also it is noteworthy that for any $R>R(\e)$, there exists $R_0\in (R, 2R)$ such that all the results remain valid. Then we consider the rescalings
\begin{align*}
    u_{R_0}(z)&:=u(R_0z), \text{ for }z\in\overline{B_1},\\
    \tilde{A}&:= (\f{x_A}{R_0},\f{y_A}{R_0}),\\
    \tilde{B}&:= (\f{x_B}{R_0},\f{y_B}{R_0}),\\
    \tilde{C}&:= (\f{x_C}{R_0},\f{y_C}{R_0}),\\
    \tilde{D}&:= (\f{x_D}{R_0},\f{y_D}{R_0}),\\
    {I}_{R_0}^i&:=\{z\in\pa B_1:\,R_0z\in I_i\},\ \ i=1,2,3,\\
    \mathcal{D}_{R_0}^i&:=\{z\in B_1:\, R_0z\in\mathcal{D}_i \},\ \ i=1,2,3,\\
    T_{R_0}& := \tilde{D}\tilde{A}\cap \tilde{D}\tilde{B}\cap \tilde{D}\tilde{C}.
\end{align*}

From Section \ref{sec: bdy data} and Proposition \ref{prop: loc of interface} we have the following properties:

\begin{enumerate}
    \item On $\pa B_1$, 
    \begin{equation*}
        |u_{R_0}(z)-a_1|\leq C\delta,\ \forall z\in I_{R_0}^i,\,i=1,2,3,
    \end{equation*}
    where the constant $C$ only depends on $W$.
    \item The diffuse interface is contained in an $O(R_0^{\beta-1})$ neighborhood of $T_{R_0}$, i.e. \begin{equation*}
        \{z\in B_1: |u_{R_0}-a_i|\geq 2C\delta,\, \forall i=1,2,3\}\subset \{z\in B_1:\, \dist(z,T_{R_0})\leq C(\d,W)R_0^{\beta-1}\}.
    \end{equation*}
    \item For $i=1,2,3$, within $\mathcal{D}_{R_0}^i$, 
    \begin{equation*}
        |u_{R_0}(z)-a_i|\leq Ke^{-kR_0 (\dist(z,\pa \mathcal{D}_{R_0}^i)-CR_0^{\beta-1})^+}.
    \end{equation*}
\end{enumerate}

Denote the approximated triple junction map associated with $T_{R_0}$ by
\begin{equation}
    U_{R_0}(z):=a_1\chi_{\mathcal{D}_{R_0}^1}(z)+a_2\chi_{\mathcal{D}_{R_0}^2}(z)+a_3\chi_{\mathcal{D}_{R_0}^3}(z), \quad z\in B_1,
\end{equation}
The aforementioned properties imply the following estimate on the $L^1$-closeness of $u_{R_0}$ and $U_{R_0}$.  

\begin{proposition}\label{prop: L1 closeness of u and U}
    There is a constant $C_1=C_1(W)$ that satisfies
    \beq\label{est: L^1 closeness}
    \|u_{R_0}-U_{R_0}\|_{L^1(B_1)}\leq C_1R_0^{\beta-1}.
    \eeq
\end{proposition}

\begin{proof}
Let $C=C(W)$ be the constant in Proposition \ref{prop: loc of interface}. Note that this constant initially depends on $\delta$ and $W$. However, since we have fixed a $\delta$ that depends on $W$, it now becomes a constant solely dependent on $W$. We have
\begin{align*}
    \int_{\mathcal{D}_{R_0}^1} |u_{R_0}(z)-U_{R_0}(z)| \,dz &=\int_{\mathcal{D}_{R_0}^1} |u_{R_0}(z)-a_1|\,dz\\
    &=\int_{\dist(z,\pa \mathcal{D}_{R_0}^1)<CR_0^{\beta-1}} |u_{R_0}(z)-a_1|\,dz+\int_{\dist(z,\pa \mathcal{D}_{R_0}^1)\geq CR_0^{\beta-1}} |u_{R_0}(z)-a_1|\,dz\\
    &=: \Lambda_1+\Lambda_2.
\end{align*}

Using $\mathcal{H}^2(\mathcal{D}_{R_0}^{1}\cap\{\dist(z,\pa \mathcal{D}_{R_0}^1)<CR_0^{\beta-1}\})\leq CR_0^{\beta-1}$, and $|u|\leq M$ by \eqref{uniform C1 bound}, we get that 
\beq\label{est: Lambda 1}
\Lambda_1\leq C(M) R_0^{\beta-1}.
\eeq

We observe that $\pa D_{R_0}^1$ consists of a circular arc $I_{R_0}^1$ and two line segments $\tilde{D}\tilde{A}$, $\tilde{D}\tilde{B}$. When $\dist(z,\pa \mathcal{D}_{R_0}^1)\geq CR_0^{\beta-1}$, 
\begin{equation*}
    \dist(z,\pa \mathcal{D}_{R_0}^1)=\min \left\{ \dist(z,\tilde{D}\tilde{A}), \dist(z,\tilde{D}\tilde{B}), \dist(z,I_{R_0}^1)\right\}.
\end{equation*}

We compute
\begin{equation*}
    \int_{\dist(z,\pa \mathcal{D}_{R_0}^1)\geq CR_0^{\beta-1}} e^{-kR_0(\dist(z, \tilde{D}\tilde{A})-CR_0^{\beta-1})}\,dz
    \leq 2\int_{0}^1 e^{-kR_0r}\,dr
    \leq \f{2}{kR_0}.
\end{equation*}

Similarly, we have
\begin{equation*}
    \int_{\dist(z,\pa \mathcal{D}_{R_0}^1)\geq CR_0^{\beta-1}} e^{-kR_0(\dist(z, \tilde{D}\tilde{B})-CR_0^{\beta-1})}\,dz
    \leq 2\int_{0}^1 e^{-kR_0r}\,dr
    \leq \f{2}{kR_0}.
\end{equation*}

\begin{equation*}
    \int_{\dist(z,\pa \mathcal{D}_{R_0}^1)\geq CR_0^{\beta-1}} e^{-kR_0(\dist(z, \pa B_1)-CR_0^{\beta-1})}\,dz
    \leq 2\pi \int_{0}^1 (1-r) e^{-kR_0r}\,dr
    \leq \f{2\pi}{kR_0}.
\end{equation*}

Summing the inequalities above, we get
\begin{equation}\label{est: Lambda 2}
\Lambda_2\leq C(k) R_0^{-1}. 
\end{equation}

Combining \eqref{est: Lambda 1} and \eqref{est: Lambda 2} yields 
\begin{equation}
    \int_{\mathcal{D}_{R_0}^1} |u_{R_0}(z)-U_{R_0}(z)|\,dz\leq CR_0^{\beta-1}. 
\end{equation}
The same estimate also applies for $\mathcal{D}_{R_0}^2$ and $\mathcal{D}_{R_0}^3$ and \eqref{est: L^1 closeness} immediately follows.    
\end{proof}

\section{Proof of Theorem \ref{main theorem}}\label{sec: conclusion}

In this final section we will conclude the proof of the main theorem. It suffices to show that the sequential limit $u_0$ in \eqref{seq conv} is unique, or equivalently, indepedent of the sequence $\{r_k\}$. 

We argue by contradiction. Suppose there are two sequences of radii $\{r_k\}_{k=1}^\infty$ and $\{s_k\}_{k=1}^\infty$ that converge to two distinct triple junction maps, i.e. 
\begin{align}
    \label{converge to u1}u_{r_k}(z):=u(r_kz)&\ri u_1(z) \text{  in }L^1_{loc}(\BR^2),\ r_k\ri\infty \text{ as }k\ri\infty,\\
    \label{converge to u2}u_{s_k}(x):=u(s_kz)&\ri u_2(z) \text{  in }L^1_{loc}(\BR^2),\ s_k\ri \infty \text{ as }k\ri\infty,
\end{align}
where $u_1=\sum\limits_{j=1}^3 a_j\chi_{\mathcal{D}_1^j} $, for $\mathcal{P}_1:=\{\mathcal{D}^1_1,\mathcal{D}_1^2,\mathcal{D}_1^3\}$ is a minimal 3-partition of $\BR^2$, with $\pa \mathcal{P}=\bigcup\limits_{j=1}^3 \pa \mathcal{D}_1^j$ consists of three rays emanating from the origin and form 120-degree angles pairwise. Similarly, $u_2$ can be represented as $u_2=\sum\limits_{j=1}^3 a_k\chi_{\mathcal{D}_2^j}$ where $\mathcal{P}_2:=\{\mathcal{D}_2^1,\mathcal{D}_2^2,\mathcal{D}_2^3\}$ is another minimal 3-partition, which centered at the origin.  %Suppose 
%\begin{equation*}
%    \|u_1-u_2\|_{L^1(B_1)}\geq \alpha>0,
%\end{equation*}
%for some positive constant $\alpha$.

For $\e\ll1$, by Section \ref{sec: existence} and Section \ref{sec: rescale} we can find a $R(\e)$ such that for any $R\geq R(\e)$, there exists $R_0\in (R,2R)$ such that the rescaled function $u_{R_0}(x)=u(R_0x)$ is close to the approximate triple junction map $U_{R_0}$ which is defined through Section \ref{sec: bdy data}, \ref{sec: loc of interface}, \ref{sec: rescale}. The $L^1$ closeness is controlled by $O(R_0^{\beta-1})$ for some $\beta\in(0,1)$. Now we fix the choice of $R_0\in (R(\e), 2R(\e))$, and then find a sequence of radii $\{R_i\}_{i=1}^\infty$ such that $R_i\in (2^i(R(\e)), 2^{i+1}R(\e))$ and $u_{R_i}$ satisfies 
\begin{equation*}
    \|u_{R_i}-U_{R_i}\|_{L^1(B_1)}\leq CR_i^{\beta-1}
\end{equation*}
for some $U_{R_i}:=\sum\limits_{j=1}^3 a_j \chi_{\mathcal{D}^j_i}$ with $\{\mathcal{D}^j_i\}$ is a partition of $B_1$ by three rays forming 120-degree angles pairwise (not necessarily centered at the origin). For each $i\in \mathbb{N}^{+}$, we denote 
\begin{align*}
    \tilde{D}_i&:= \bigcap\limits_{j=1}^3\pa \mathcal{D}_i^j\ \text{  is the junction point for three phases of }U_{R_i},\\
    \tilde{A}_i&:= \pa B_1\cap \pa \mathcal{D}_i^1 \cap \pa \mathcal{D}_i^2,\\
    \tilde{B}_i&:= \pa B_1\cap \pa \mathcal{D}_i^1 \cap \pa \mathcal{D}_i^3,\\
    \tilde{C}_i&:= \pa B_1\cap \pa \mathcal{D}_i^2 \cap \pa \mathcal{D}_i^3,\\
    T_i&:= \tilde{D}_i\tilde{A}_i\cap \tilde{D}_i\tilde{B}_i\cap \tilde{D}_i\tilde{C}_i \ \text{ is the triod at the scale $R_i$}. 
\end{align*}
The construction of this $U_{R_i}$ proceeds in the same manner as in the beginning of Section \ref{sec: lower upper bdd}. We first choose $A_i,B_i,C_i$ as the middle points of the three transitional arcs between different phases on $\pa B_{R_i}$, followed by the selection of $D_i$ such that $|D_iA_i|+|D_iB_i|+|D_iC_i|$ is minimized. Then we  divide the coordinates by $R_i$ to rescale all the points $(A_i, B_i, C_i, D_i)$ to $(\tilde{A}_i, \tilde{B}_i,\tilde{C}_i,\tilde{D}_i)\in \overline{B}_1$. All the earlier results presented in Section \ref{sec: bdy data}, \ref{sec: lower upper bdd}, \ref{sec: loc of interface}, \ref{sec: rescale} remain valid upon substituting $(A,B,C,D)$ with $(A_i, B_i, C_i, D_i)$ respectively, and replacing $R_0$ by $R_i$.   

For each $i$, we denote by $\theta_i\in(0,2\pi]$ the angle of the vector $\overrightarrow{D_iA_i}$, which represents the direction of the $a_1$-$a_2$ interface for the approximate triple junction map at scale $R_i$. Given that the angles between each pair of interfaces are all 120 degrees, $U_{R_i}$ is completely determined by $\theta_i$ and the coordinates of $D_i$. The following lemma says that in two consecutive scales, $\theta_i$ won't change too much. 

\begin{lemma}\label{lemma: difference of theta i}
    There is a constant $C=C(W)$ such that for any $i\in\mathbb{N}^+$,
    \beq\label{difference of theta i}
    |\theta_i-\theta_{i+1}|\leq CR_i^{\beta-1}.
    \eeq
\end{lemma}

\begin{proof}
    %idea: fix D_{i+1} as origin and A_{i+1}=(1,0). Then use the corollary to show $D_i$ is very close to origin. 

    Without loss of generality, we assume $\tilde{D}_{i+1}=(0,0)$ and $\tilde{A}_{i+1}=(0,1)$, which means the approximated $a_1$-$a_2$ interface at scale $R_{i+1}$ aligns with $y$-axis.  Set the coordinates
    \begin{equation*}
        \begin{split}
        &D_{i}=(x_D^i,y_D^i), \ \ A_{i}=(x_A^i, y_A^i),\\
        &\tilde{D}_{i}=\f{1}{R_i}(x_D^i,y_D^i), \ \ \tilde{A}_{i}=\f{1}{R_i}(x_A^i, y_A^i).
        \end{split}
    \end{equation*}

    According to Corollary \ref{corol: D is close to origin} we have 
    \begin{equation*}
        \sqrt{(x_D^i)^2+(y_D^i)^2}\leq \tau R_i, \quad \text{ for some }\tau\leq \f{1}{100}. 
    \end{equation*}

    We consider the original minimizer $u$ on $B_{R_{i+1}}$. By Proposition \ref{prop: loc of interface},
    \begin{equation}\label{u close to a1 or a2 on each sides of y axis}
    \begin{split}
        &|u(x,y)-a_1|\leq 2C\delta,\quad \text{if }x\leq -C_0R_{i+1}^\beta,\ y\geq C_0R_{i+1}^\b,\ (x,y)\in B_{R_{i+1}},\\
        &|u(x,y)-a_2|\leq 2C\delta,\quad \text{if }x\geq C_0R_{i+1}^\beta,\ y\geq C_0R_{i+1}^\b,\ (x,y)\in B_{R_{i+1}},
    \end{split}
    \end{equation}
    for some constant $C_0(W)$.

    Since $R_i\in (2^iR(\e), 2^{i+1}R(\e))$, we have 
    \beqo
    \f14< \f{R_i}{R_{i+1}}<1. 
    \eeqo
    Note that rescaling does not affect the direction of straight lines. Thus in order to show \eqref{difference of theta i}, it suffices to prove that three exists a constant $C$ such that \begin{equation}\label{slope of DiAi}
        |x_A^i-x_D^i|\leq CR_i^{\beta}. 
    \end{equation}
    Now we take $C=48C_0$ and show that violation of \eqref{slope of DiAi} will yield a contradiction. Suppose $|x_A^i-x_D^i|> 48C_0R_i^{\beta}.$ Then it holds that  
    \beqo
    |\f{2x_{A}^i+x_{D}^i}{3}-\f{x_A^i+2x_D^i}{3}|> 16 C_0R_i^{\beta},
    \eeqo
    which further implies that the distance of either $(\frac{2x_A^i+x_D^i}{3},\frac{2y_A^i+y_D^i}{3})$ or $(\frac{x_A^i+2x_D^i}{3},\frac{y_A^i+2y_D^i}{3})$ (which are two trisection points of the line segment $D_iA_i$) from the $y$-axis exceeds $8C_0R_i^{\beta}$.

    Suppose $\frac{2x_A^i+x_D^i}{3}>8C_0R_i^\beta$ (the case $\frac{2x_A^i+x_D^i}{3}<-8C_0R_i^\beta$ or $|\f{x_A^i+2x_D^i}{3}|> 8C_0R_i^\beta$ can be analyzed similarly). We set
    \begin{equation*}
        E:=(5C_0R_i^{\beta}, \frac{2y_A^i+y_D^i}{3}).
    \end{equation*}
    According to \eqref{u close to a1 or a2 on each sides of y axis}, we have $|u(E)-a_2|\leq 2C(W)\delta$ because $5C_0R_i^{\beta}>C_0R_{i+1}^\beta$. On the other hand, we also have 
    \begin{equation*}
        5C_0R_i^{\beta}-\frac{2x_A^i+x_D^i}{3}\leq -3C_0R_i^{\beta},
    \end{equation*}
    which implies that $|u(E)-a_1|\leq 2C(W)\delta$ through a similar estimate as in \eqref{u close to a1 or a2 on each sides of y axis} at scale $R_i$. Consequently, we arrive at a contradiction (as illuatrated in Figure \ref{fig: consecutive scales}), thereby concluding the proof.
\end{proof}

\begin{figure}[ht]
\centering
\begin{tikzpicture}[thick]

\fill[color=black!20] (-0.15,0)--(-0.15,6) --(0.15, 6)--(0.15, 0)--(-0.15,0);

\fill[color=black!20] (-0.33,0.25)--(0.97,3.15) --(1.23, 3.05)--(-0.07, 0.15)--(-0.33,0.25);

\filldraw (0,0) circle (0.03cm) node[below right]{\small{$D_{i+1}=(0,0)$}};
\filldraw (0,6) circle (0.03cm) node[above]{\small{$A_{i+1}=(0,R_{i+1})$}};
\draw (0,0)--(0,6);
\filldraw (-0.2,0.2) circle (0.03cm) node[below left]{$D_{i}$};
\filldraw (1.1,3.1) circle (0.03cm) node[above right]{$A_{i}$};
\draw (-0.2,0.2)--(1.1,3.1);

\draw[dotted] (-0.8,2.2)--(1.5,2.2);
\filldraw (0.4,2.2) circle (0.03cm) node[above]{\small{$E$}};

\end{tikzpicture}
\caption{Grey region: $O(R_i^\beta)$ transition layer between $a_1$ and $a_2$. A significant difference between $\t_i$ and $\t_{i+1}$ leads to a contradiction that $E$ belongs to both the $a_1$ phase and the $a_2$ phase. }
\label{fig: consecutive scales}
\end{figure}
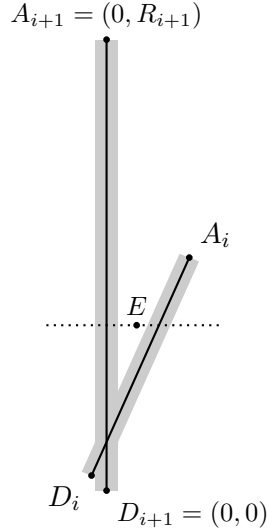

Since $\mathcal{P}_1$ and $\mathcal{P}_2$ are two distinct triple junctions centered at the origin, their directions of $a_1$-$a_2$ interface should be different, i.e. 
\begin{equation*}
    \varphi_1-\varphi_2=\al, \quad \text{for some angle }\al\neq 0.  
\end{equation*}
Here $\varphi_i$ denotes the angle of $\pa \mathcal{D}_i^1\cap \pa \mathcal{D}_i^2$ for $i=1,2$.

For each sufficiently large $k$, there exist $i(k), j(k)\in \mathbb{N}^+$ such that 
\beqo
r_k\in [R_{i(k)}, R_{i(k)+1}],\quad s_k\in [R_{j(k)}, R_{j(k)+1}],
\eeqo
with $i(k)$ and $j(k)$ tend to infinity as $k$ increases.

On $B_1$, $u_{r_k}$ is converging to a triple junction map $u_1$ thanks to \eqref{converge to u1}; while for the comparable scaling $u_{R_{i(k)}}$ is close to another triple junction map $U_{R_{i(k)}}$. By the same argument as in the proof of Lemma \ref{lemma: difference of theta i} we can conclude that
\begin{align*}
&\lim\limits_{k\ri\infty}\|U_{R_{i(k)}}-u_1\|_{L^1(B_1)}=0,\\
&\lim\limits_{k\ri\infty}|\t_{i(k)}-\varphi_1|=0.
\end{align*}
Similarly, we have
\begin{align*}
&\lim\limits_{k\ri\infty}\|U_{R_{j(k)}}-u_2\|_{L^1(B_1)}=0,\\
&\lim\limits_{k\ri\infty}|\t_{j(k)}-\varphi_2|=0.
\end{align*}

Therefore we obtain two subsequences of $\t_{i}$ that converge to distinct angles. However, by Lemma \ref{lemma: difference of theta i}, we have for any $i<j$, 
\begin{equation*}
\begin{split}
|\t_i-\t_j|&\leq \sum\limits_{l=i}^{j-1} CR_l^{\beta-1}\\
&\leq \sum\limits_{l=i}^{j-1} C\left(2^l R(\e)\right)^{\beta-1}\\
&= C2^{i(\beta-1)}R(\e)^{\beta-1} \sum\limits_{l=0}^{j-i-1} 2^{l(\beta-1)}\\
&\leq \f{C2^{i(\beta-1)} R(\e)^{\beta-1}}{1-2^{\beta-1}} \ri 0 \ \text{ as }i\ri\infty,
\end{split}
\end{equation*}
indicating that $\{\t_i\}$ is a Cauchy sequence, which yields a contradiction. The proof of Theorem \ref{main theorem} is complete. 

% add a drawing here.

\appendix

\section{Minimization problem about $y^*$} \label{minimization y*}

Recall the minimization problem \eqref{minimization problem: y*}:
\begin{equation*}
    \min\quad f(y^*)=(y_A'-y^*)+ \sqrt{(x_C'-x_B')^2+(2y^*-y_B'-y_C')^2},
\end{equation*}
subject to $y^*\in (y_C'+d_2,y_A'-d_2)$. Taking derivatives with respect to $y^*$,
\begin{align*}
    \f{\pa f}{\pa y^*}&=-1+\f{2(2y^*-y_B'-y_C')}{\sqrt{(x_C'-x_B')^2+(2y^*-y_B'-y_C')^2}},\\
    \f{\pa^2 f}{\pa {y^*}^2} &=\f{4(x_C'-x_B')^2}{((x_C'-x_B')^2+(2y^*-y_B'-y_C')^2)^{\f23}}>0.
\end{align*}

Form the expression above we have that when $y\sim y_C'$, $\f{\pa f}{\pa y^*}<0$, whereas when $y\sim y_A'$, $\f{\pa f}{\pa y^*}>0$. Therefore there is only one critical point for $f(y^*)$ on the interval $(y_C'+d_2,y_A'-d_2)$, which is a minimum point. Moreover, the minimum $y^*_{min}$ satisfies
\begin{equation*}
    3(2y^*_{min}-y_B'-y_C')^2=(x_C'-x_B')^2
\end{equation*}

Since $\overrightarrow{DB'}$ parallels $(-\sqrt{3},-1)$ and $\overrightarrow{DC'}$ parallels $(\sqrt{3},-1)$,
\begin{equation*}
    \frac{2y_D-y_B'-y_C'}{x_C'-x_B'}=\f{1}{\sqrt{3}},
\end{equation*}
which immediately implies $y^*_{min}=y_D$. 

Finally we compute
\begin{align*}
    |DA|&=y_A'-y_D,\\
    |DB|+|DC|&=2|y_D-y_B'|+2|y_D-y_C'|=\sqrt{(x_C'-x_B')^2+(2y_D-y_B'-y_C')^2}.\\
\end{align*}
Hence,
\begin{equation*}
    \min\ f(y^*)=f(y_D)=|DA|+|DB|+|DC|.
\end{equation*}

\bibliographystyle{acm}
\bibliography{uniqueness}

\end{document}